\documentclass[12pt]{article}

\pdfoutput=1

\usepackage[a4paper]{geometry}
\usepackage{amsthm}
\usepackage{amsmath}
\usepackage{amsfonts}
\usepackage{amssymb}
\usepackage{graphicx}
\usepackage[colorlinks=true,citecolor=black,linkcolor=black,urlcolor=blue]{hyperref}
\newcommand{\doi}[1]{\href{http://dx.doi.org/#1}{\texttt{doi:#1}}}

\theoremstyle{plain}
\newtheorem{theorem}{Theorem}
\newtheorem{lemma}[theorem]{Lemma}
\newtheorem{corollary}[theorem]{Corollary}
\newtheorem{proposition}[theorem]{Proposition}

\theoremstyle{definition}
\newtheorem{definition}[theorem]{Definition}

\theoremstyle{remark}

\renewcommand{\emptyset}{\varnothing}

\title{Isotropic matroids II\@: Circle graphs}

\author{Robert Brijder\thanks{R.B.\ is a postdoctoral fellow of the Research Foundation -- Flanders (FWO).}\\
\small Hasselt University\\[-0.8ex]
\small Belgium\\
\small\tt robert.brijder@uhasselt.be\\
\and
Lorenzo Traldi\\
\small Lafayette College\\[-0.8ex]
\small Easton, Pennsylvania, U.S.A.\\
\small\tt traldil@lafayette.edu
}

\date{}

\begin{document}

\maketitle

\begin{abstract}
We present several characterizations of circle graphs, which follow from
Bouchet's circle graph obstructions theorem.

\end{abstract}

\section{Introduction}

Let $F$ be a 4-regular graph and let $C$ be an \emph{Euler system} of $F$,
i.e., a set that includes precisely one Euler circuit of each connected
component of $F$. Then the \emph{interlacement graph} $\mathcal{I}(C)$ of $C$ is the
simple graph with vertex-set equal to the set $V(F)$ of vertices of $F$, in which $v_{i}$ and $v_{j}$ are adjacent
if and only if they are \emph{interlaced} with respect to $C$, i.e., they
appear in the order $v_{i} \ldots v_{j} \ldots v_{i} \ldots v_{j}$ or $v_{j} \ldots v_{i} \ldots v_{j} \ldots v_{i}$ on one of the circuits of $C$. A simple graph that arises
from this construction is called a \emph{circle graph}.

The idea of interlacement is almost 100 years old, as it was used by Brahana
in defining his separation matrix \cite{Br}. Interlacement graphs were first
discussed by Zelinka \cite{Z}, who credited the idea to Kotzig. But circle
graphs did not become well known until the 1970s, when Cohn and Lempel
\cite{CL} and Even and Itai \cite{EI} used them to analyze permutations, and
Bouchet \cite{Bold} and Read and Rosenstiehl \cite{RR} used them to study
Gauss' problem of characterizing generic self-intersecting curves in the
plane. Circle graphs were studied intensively during the next few decades.
Among the notable results of this intensive study are polynomial-time
recognition algorithms due to Bouchet \cite{Bec}, Gioan, Paul, Tedder and
Corneil \cite{GPTC}, Naji \cite{N} and Spinrad \cite{Sp}. In recent years, interest in 4-regular graphs (and indirectly on circle graphs) has also focused on their appearance as medials of graphs imbedded on surfaces~\cite{EMM1}.

See Figure \ref{circmf8} for an example. On the left is a 4-regular graph,
with an indicated Euler circuit. To trace the Euler circuit just
walk along the edges, making sure to preserve the dash style (dashed or plain)
when traversing a vertex; the dash style will sometimes change in the middle
of an edge, though. On the right is the resulting interlacement graph.%

\begin{figure}
[ptb]
\begin{center}
\includegraphics[
trim=1.333998in 7.762249in 0.672097in 1.075247in,
height=1.6535in,
width=4.8948in
]%
{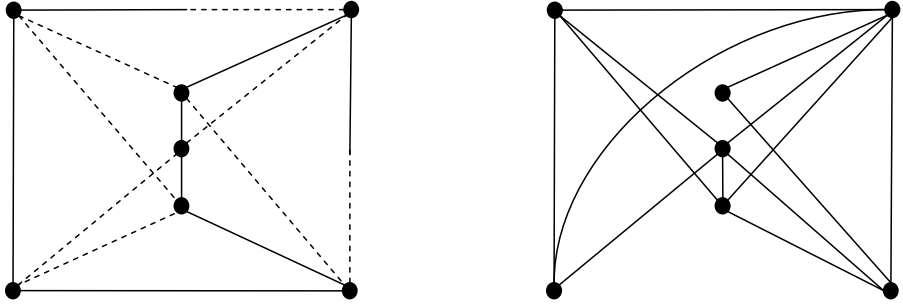}%
\caption{An Euler circuit and its interlacement graph.}%
\label{circmf8}%
\end{center}
\end{figure}

If $v\in V(F)$, then the $\kappa$\emph{-transform} $C\ast v$ is the Euler
system obtained from $C$ by reversing one of the $v$-to-$v$ walks within the
circuit of $C$ incident at $v$. As we do not distinguish between circuits that
differ only in starting point or orientation, the same Euler system will
result no matter which of the two $v$-to-$v$ walks is reversed. The $\kappa
$-transformations were introduced by Kotzig \cite{K}, who proved the
fundamental theorem that any two Euler systems of $F$ are connected through
some finite sequence of $\kappa$-transformations. As noted by Read and
Rosenstiehl \cite{RR}, the interlacement graph $\mathcal{I}(C\ast v)$ is the
\emph{simple local complement }$\mathcal{I}(C)_{s}^{v}$, the simple graph
obtained from $\mathcal{I}(C)$ by reversing all adjacencies among neighbors of $v$. Simple graphs that can be obtained from each other
through local complementations are said to be \emph{locally equivalent}, and an induced subgraph of a simple graph locally equivalent to $G$ is called a \emph{vertex-minor} of $G$.

We use the subscript $s$ to distinguish simple local complementation from a closely related operation $G \mapsto G_{ns}^{v}$ on looped simple graphs, which we call \emph{non-simple} local complementation. This operation reverses the adjacencies between every pair of vertices in the open neighborhood $N(v)$ and complements the loop status of every vertex in $N(v)$. Looped simple graphs that can be obtained from each other through non-simple local complementations and loop complementations are said to be locally equivalent. In particular, $G$ is a circle graph if and only if every graph obtained from $G$ using loop complementations is also a circle graph; consequently loops are irrelevant to characterizations of circle graphs, and most results regarding circle graphs are stated for simple graphs.

Bouchet \cite{Bco} gave a famous characterization of circle graphs: a simple graph is a circle graph if and only if none of the
three graphs pictured in Figure \ref{circmf4} is a vertex-minor. We refer to this famous result as \emph{Bouchet's theorem}. Bouchet's theorem resembles several well-known forbidden minors
characterizations of matroid classes: for instance a matroid is binary iff
$U_{2,4}$ is not a minor, a binary matroid is regular iff neither $F_{7}$ nor
$F_{7}^{\ast}$ is a minor, and a regular matroid is graphic iff neither
$M^{\ast}(K_{5})$ nor $M^{\ast}(K_{3,3})$ is a minor. But Bouchet's theorem
involves induced subgraphs rather than matroid minors, and including local
equivalence makes Bouchet's theorem seem more complicated than the classic
matroid results. The present paper was initially motivated by a couple of
questions suggested by this resemblance: Can Bouchet's theorem be rephrased to
characterize circle graphs using matroids? If so, is it possible to state such a characterization
without mentioning local equivalence? It turns out that both answers are
``yes''; we present several such
characterizations below. In the process of explaining them we also obtain
other circle graph characterizations, some of which involve local equivalence
and do not explicitly mention matroids.%

\begin{figure}
[ptb]
\begin{center}
\includegraphics[
trim=1.868446in 8.300424in 0.670397in 1.078549in,
height=1.2462in,
width=4.5065in
]%
{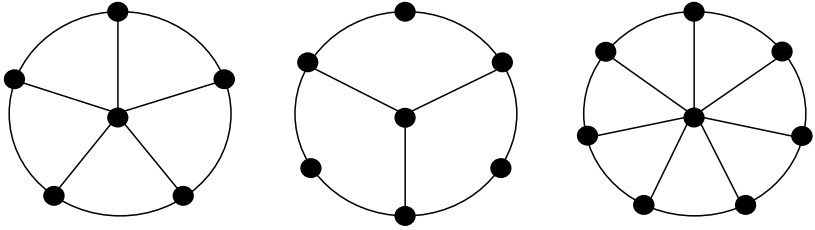}%
\caption{Bouchet's circle graph obstructions: $W_{5}$, $BW_{3}$ and $W_{7}$.}%
\label{circmf4}%
\end{center}
\end{figure}

To state these characterizations, we introduce some terminology. First, we note that we follow the convention that if $X$ and $Y$ are finite sets then an $X \times Y$ matrix has rows indexed by elements of $X$ and columns indexed by elements of $Y$. If $G$ is a simple graph then we consider the adjacency matrix $A=A(G)$ and the $n \times n$ identity matrix $I$ as $V(G) \times V(G)$ matrices over $GF(2)$. Let $IAS(G)$ be the matrix
\[
IAS(G)=%
\begin{pmatrix}
I & A & I+A
\end{pmatrix}
\text{.}%
\]
The rows of $IAS(G)$ inherit indices in $V(G)$ from the rows of $I$ and $A(G)$. Notation for the columns of $IAS(G)$ follows this
scheme: the $v$ column of $I$ is designated $\phi_{G}(v)$, the $v$ column of
$A$ is designated $\chi_{G}(v)$, and the $v$ column of $I+A$ is designated
$\psi_{G}(v)$. The set $\{\phi_{G}(v),\chi_{G}(v),\psi_{G}(v)\mid v\in V(G)\}$
is denoted $W(G)$, and the binary matroid on $W(G)$ represented by $IAS(G)$ is
the \emph{isotropic matroid} of $G$,\emph{ }$M[IAS(G)]$ \cite{Tnewnew}. If
$v\in V(G)$ then the subset $\{\phi_{G}(v),\chi_{G}(v),\psi_{G}(v)\}$ of
$W(G)$ is the \emph{vertex triple} corresponding to $v$. Notice that the three
columns of $IAS(G)$ corresponding to a vertex triple sum to $0$. If $v$ is not
isolated then each of these columns has a nonzero entry, so the vertex triple
is a circuit of $M[IAS(G)]$. If $v$ is isolated, instead, then $\{\chi
_{G}(v)\}$ and $\{\phi_{G}(v),\psi_{G}(v)\}$ are separate circuits of
$M[IAS(G)]$. A \emph{transversal} of $W(G)$ is a subset that includes
precisely one element of each vertex triple, and a subset of a transversal is
a \emph{subtransversal}. A \emph{transverse matroid} of $G$ is a matroid
obtained by restricting $M[IAS(G)]$ to a transversal. (We use
``transverse matroid'' to avoid confusion
with transversal matroids.) A \emph{transverse circuit} of $G$ is a circuit of a transverse matroid of $G$.

The\ general theory of isotropic matroids is presented in \cite{BT1} and
\cite{Tnewnew}. Part of this theory is the following result:

\begin{theorem}
\label{theory}(\cite{BT1}) Let $G$ and $H$ be simple graphs. Then any one of the
following conditions implies the others:

\begin{enumerate}
\item $G$ and $H$ are locally equivalent, up to isomorphism.

\item The isotropic matroids of $G$ and $H$ are isomorphic.

\item There is a bijection between $W(G)$ and $W(H)$, which defines
isomorphisms between the transverse matroids of $G$ and those of $H$.

\item There is a bijection between $W(G)$ and $W(H)$, under which vertex
triples and transverse circuits of $G$ and $H$ correspond.
\end{enumerate}
\end{theorem}

In particular, if $G$ and $H$ are locally equivalent then each sequence of
local complementations that may be used to obtain $H$ from $G$ yields an
\emph{induced isomorphism} $M[IAS(G)]\rightarrow M[IAS(H)]$, which is
compatible with the partitions of $W(G)$ and $W(H)$ into vertex triples.

All the interlacement graphs of Euler systems of a particular 4-regular graph
$F$ are locally equivalent (up to isomorphism); it follows that the class of circle graphs is
closed under local complementation. Theorem \ref{theory} then implies that
there must be matroidal characterizations of circle graphs using their
isotropic matroids, their transverse circuits and their transverse matroids.
Circle graph characterizations involving isotropic matroids are complicated by
the fact that the class of isotropic matroids is not closed under matroid
minors. (The order of an isotropic matroid is divisible by 3, so deleting or
contracting an element of an isotropic matroid cannot yield another isotropic
matroid.) In order to derive such characterizations we need a special minor
operation that is appropriate for isotropic matroids.

\begin{definition}
\label{isominor}Let $G$ be a looped simple graph, let $S$ be a subtransversal
of $W(G)$, and let $S^{\prime}$ contain the other $2\left\vert S\right\vert $
elements of $W(G)$ that\ correspond to the same vertices of $G$ as elements of
$S$. Then the \emph{isotropic minor }of $G$ obtained by contracting $S$ is the
matroid%
\[
(M[IAS(G)]/S)-S^{\prime}.
\]

\end{definition}

We use the term \emph{isotropic minor} because the definition is consistent
with Bouchet's definitions of minors of isotropic systems \cite{Bi1} and
multimatroids \cite{B2}.

\begin{theorem}
(\cite{Tnewnew}) The isotropic minors of $G$ are precisely the isotropic
matroids of vertex-minors of $G$.
\end{theorem}

Bouchet's theorem now leads directly to a characterization of circle graphs by
excluded isotropic minors.

\begin{theorem}
\label{bouchet}Let $G$ be a simple graph. Then $G$ is a circle graph if, and
only if, the isotropic matroids of $W_{5}$, $BW_{3}$ and $W_{7}$ are not
isotropic minors of $G$.
\end{theorem}

It follows that we may try to gain insight into the special characteristics of
circle graphs by contrasting their isotropic minors of size $\leq24$ with the
isotropic matroids of $W_{5}$, $BW_{3}$ and $W_{7}$. Formulating and verifying
these contrasts is facilitated by the following four theorems, which show that
isotropic matroids reflect important structural properties of graphs.

\begin{theorem}
\label{smallcir2}Let $G$ be an interlacement graph of an Euler system of a
4-regular graph $F$, and let $k$ be a positive integer. If $F$ has a
$k$-circuit then $G$ has a transverse circuit of size $\leq k$.
\end{theorem}

Theorem \ref{smallcir2} is discussed in Section~\ref{sec:4reg_graphs}. The next three theorems are discussed in \cite{BT1}.

\begin{theorem}
\label{smallcir1}(\cite{BT1}) Let $G$ be a simple graph, and let $k$ be a
positive integer. Then $G$ has a transverse $k$-circuit if and only if some
graph locally equivalent to $G$ has a vertex of degree $k-1$.
\end{theorem}

\begin{theorem}
\label{smallcir3}(\cite{BT1}) Let $G$ be a simple graph, and let $k_{1},k_{2}$
be positive integers. Then $G$ is locally equivalent to a graph with adjacent
vertices of degrees $k_{1}-1$ and $k_{2}-1$ if and only if $G$ has transverse
circuits $\gamma_{1},\gamma_{2}$ such that $\left\vert \gamma_{i}\right\vert
=k_{i}$, the largest subtransversals contained in $\gamma_{1}\cup\gamma_{2}$
are of cardinality $\left\vert \gamma_{1}\cup\gamma_{2}\right\vert -2$, and
two of these largest subtransversals are independent sets of $M[IAS(G)]$.
\end{theorem}

\begin{theorem}
\label{smallcir4}(\cite{BT1}) Let $G$ be a simple graph, and let $k_{1},k_{2}$
be positive integers. Then these two conditions are equivalent:

\begin{itemize}
\item $G$ is locally equivalent to a graph with nonadjacent vertices of
degrees $k_{1}-1$ and $k_{2}-1$, which share no neighbor.

\item $G$ has disjoint transverse circuits $\gamma_{1},\gamma_{2}$ such that
$\left\vert \gamma_{i}\right\vert =k_{i}$ and $\gamma_{1}\cup\gamma_{2}$ is a
subtransversal, which contains no other circuit.
\end{itemize}
\end{theorem}

Here is an illustration of the usefulness of these properties. It is easy to
see that up to isomorphism, there are only two simple 4-regular graphs with
$\leq6$ vertices: one is $K_{5}$ and the other is obtained from $K_{6}$ by
removing the edges of a perfect matching. Each of these graphs contains
several 3-circuits. A non-simple 4-regular graph must contain a 1-circuit or a
2-circuit, of course, so Theorem \ref{smallcir2} tells us that every circle
graph with $\leq6$ vertices has a transverse circuit of size $\leq3$.
According to Theorem \ref{smallcir1}, this is equivalent to saying that every
circle graph with $\leq6$ vertices is locally equivalent to a graph with a
vertex of degree $\leq2$. Inspecting the matrix $IAS(W_{5})$, it is not hard
to see that the only circuits of size $\leq3$ in $M[IAS(W_{5})]$\ are vertex
triples; the smallest transverse circuits are of size 4. It follows from
Theorem \ref{smallcir1} that no simple graph locally equivalent to $W_{5}$ has
a vertex of degree $\leq2$. This is enough to verify the following.

\begin{corollary}
\label{smallchar}Let $G$ be a simple graph with $\leq6$ vertices. Then any one
of the following properties implies the others.

\begin{enumerate}
\item $G$ is a circle graph.

\item $G$ has a transverse circuit of size $\leq3$.

\item $G$ is locally equivalent to a graph with a vertex of degree $\leq2$.
\end{enumerate}
\end{corollary}

Further analysis leads to several characterizations of larger circle graphs. For instance, here is a characterization that involves both isotropic matroids and transverse matroids.

\begin{theorem}
\label{jaeger2}A simple graph $G$ is a circle graph if and only if $G$
satisfies all of the following conditions.

\begin{enumerate}
\item Every transverse matroid of $G$ is cographic.

\item Every isotropic minor of $G$ of size $<24$ has a loop or a pair of
intersecting 3-circuits.

\item Suppose an isotropic minor $M$ of $G$ of size $24$ has no loop and no
pair of intersecting 3-circuits. Then every transverse matroid of $M$ that
contains two disjoint circuits also contains other circuits.
\end{enumerate}
\end{theorem}

The three conditions of Theorem \ref{jaeger2} correspond directly to the three
obstructions of Bouchet's theorem: condition 1 excludes $BW_{3}$, condition 2
excludes $W_{5}$, and condition 3 excludes $W_{7}$. By the way, we state
condition 2 this way only for variety. As vertex triples are dependent sets,
it is not hard to see that an isotropic matroid has a transverse circuit of
size $\leq3$ if and only if it has a loop or a pair of intersecting
3-circuits. Details of the argument appear in the proof of Corollary
\ref{jaeger1}.

\begin{figure}
[ptb]
\begin{center}
\includegraphics[
trim=1.872695in 8.279513in 2.409692in 1.081850in,
height=1.2626in,
width=3.1877in
]%
{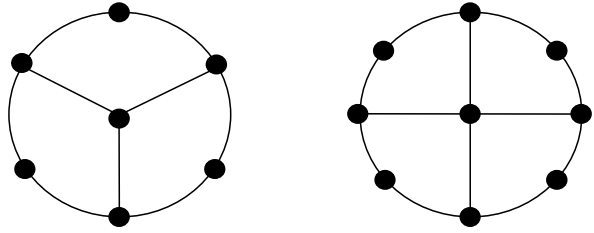}%
\caption{$BW_{3}$ and $BW_{4}$.}%
\label{circmf7}%
\end{center}
\end{figure}

Condition 1 is of particular interest for several reasons. One reason is
simply that the cographic property is more familiar than the small-circuit
properties mentioned in conditions 2 and 3. Another reason is that it is
possible to explicitly construct the graphs whose cocycle matroids are the
transverse matroids of a circle graph; see Section~\ref{sec:4reg_graphs} for details. Yet another
reason is that, as we show in Section~\ref{sec:bipartite_circle}, condition 1 suffices to characterize a special type of circle graph.

\begin{theorem}
\label{bipartite}Let $G$ be a simple graph that is locally equivalent to a
bipartite graph. Then any one of the following properties implies the others.

\begin{enumerate}
\item $G$ is a circle graph.

\item Every transverse matroid of $G$ is cographic.

\item Neither graph of Figure \ref{circmf7} is a vertex-minor of $G$.
\end{enumerate}
\end{theorem}

Bipartite circle graphs are important for two reasons. One reason is that bipartite circle graphs correspond to planar 4-regular graphs, and the other reason is that all circle graphs are vertex-minors of bipartite circle graphs. Details are given in Sections~\ref{sec:bipartite_circle} and \ref{sec:cross_nr_4reg}, along with some results about the connection between the crossing number of a 4-regular graph and the matroidal properties of its associated circle graphs.

Returning to the general case, note that conditions 2 and 3 of Theorem
\ref{jaeger2} indicate characteristic properties of the transverse circuits of small isotropic minors of circle graphs. Theorems \ref{smallcir1} -- \ref{smallcir4} tell us that these properties are related to the distribution of low-degree vertices in small vertex-minors of circle graphs.

\begin{theorem}
\label{jaeger6}Let $G$ be a simple graph, and let $\mathcal{VM}_{8}(G)$ denote
the set of graphs with 8 or fewer vertices, which are vertex-minors of $G$.
Then $G$ is a circle graph if and only if every $H\in\mathcal{VM}_{8}(G)$
satisfies at least one of the following conditions.

\begin{enumerate}
\item Some graph locally equivalent to $H$ has a vertex of degree $0$ or $1$.

\item Some graph locally equivalent to $H$ has a pair of adjacent degree-$2$ vertices.

\item Every graph locally equivalent to $H$ has a vertex of degree $5$.
\end{enumerate}
\end{theorem}

As every $H\in\mathcal{VM}_{8}(G)$ must satisfy one of the
conditions, condition 3 is logically equivalent to the simpler requirement
that $H$ itself must have a vertex of degree 5. Condition 3 may also be
replaced by the weaker requirement that there be a vertex of degree $\geq4$,
because $W_{5}$ and $W_{7}$ are both locally equivalent to 3-regular graphs.
See Figure \ref{circmf20}.%

\begin{figure}
[ptb]
\begin{center}
\includegraphics[
trim=2.004395in 8.296021in 2.276293in 1.070845in,
height=1.2566in,
width=3.1903in
]%
{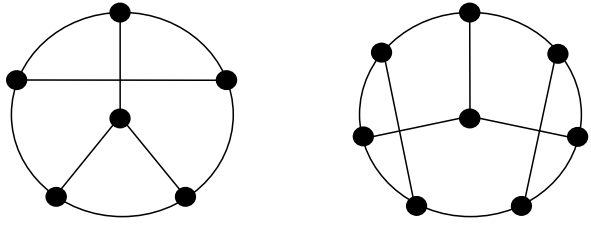}%
\caption{Cubic graphs locally equivalent to $W_{5}$ and $W_{7}$.}%
\label{circmf20}%
\end{center}
\end{figure}

Several other circle graph characterizations are presented in Sections~\ref{sec:circle_small_circuits}--\ref{sec:cross_nr_4reg}.
Although the details differ, most are variations on the theme
\textquotedblleft circle graphs have vertex-minors with distinctive
distributions of small transverse circuits.\textquotedblright\

Before deriving these matroidal characterizations of circle graphs, we discuss
a different kind of structural characterization of circle graphs, using
delta-matroids. It is shown in \cite{GeelenPhD} that a variant of
unimodularity called \emph{principal unimodularity} precisely corresponds to
representability of a delta-matroid $D$ over every field. This specializes to
the usual notion of unimodularity in case $D$ is a matroid. In this way, this
result generalizes the well-known result that unimodular representations of
matroids correspond precisely to matroids representable over every field. It
has been shown by Bouchet (see Geelen's PhD thesis \cite{GeelenPhD}) that
circle graphs are precisely the graphs $G$ such that for every graph
$G^{\prime}$ locally equivalent to $G$, $G^{\prime}$ allows for a principal
unimodular representation. In Section~\ref{sec:char_circle_dm_regularity} we
reformulate this characterization in terms of binary delta-matroids. Using
this characterization, we provide in
Section~\ref{sec:char_circle_mm_regularity} a characterization of isotropic
matroids of circle graphs in terms of principal unimodularity, but
\emph{without} mentioning local equivalence. The techniques used in
Section~\ref{sec:char_circle_mm_regularity} are from multimatroid theory
\cite{B1} and it is essentially shown that the natural $3$-matroid
generalization of principal unimodularity precisely characterizes isotropic
matroids of circle graphs.

More specifically, for an isotropic matroid $M[IAS(G)]$ and transversal $T$,
we say that $M[IAS(G)]-T$ is \emph{tight} if for every independent
subtransversal $S$ disjoint from $T$ with $|S|=|V(G)|-1$ there is an element
$x\in W(G)\setminus T$ such that $S\cup\{x\}$ is a dependent subtransversal.
Also, we say that $M[IAS(G)]-T$ is \emph{t-regular} (short for
\textquotedblleft transversal-regular\textquotedblright) if $M[IAS(G)]-T$ has
a representation $E$ such that for each transversal $T^{\prime}$ disjoint from
$T$, the determinant of the matrix obtained from $E$ by restricting to the
columns of $T^{\prime}$ is equal to $0$, $1$, or $-1$.

We show the following (cf.\ Theorem~\ref{thm:eulerian_dm_char_ias}).

\begin{theorem}
Let $G$ be a simple graph. Then $G$ is a circle graph if and only if for all
transversals $T$, if $M[IAS(G)] - T$ is tight, then $M[IAS(G)] - T$ is t-regular.
\end{theorem}

\section{Characterizing circle graphs by delta-matroid regularity}

\label{sec:char_circle_dm_regularity} In this section we recall a
characterization of circle graphs using the notion of regularity for
delta-matroids and multimatroids.

\subsection{Delta-matroids}

A \emph{set system} (over $V$) is a tuple $D=(V,S)$ such that $S\subseteq
2^{V}$ is a set of subsets of a \emph{ground set} $V$ of $D$. For notational
simplicity we write $X\in D$ to denote $X\in S$. We say that $D$ is
\emph{empty} if $S=\emptyset$. A \emph{delta-matroid} $D$ is a nonempty set
system $(V,S)$ that satisfies the following property: for all $X,Y\in S$ and
$x\in X\mathop{\mathrm{\Delta}}Y$, there is a $y\in
X\mathop{\mathrm{\Delta}}Y$ (we allow $y=x$) such that
$X\mathop{\mathrm{\Delta}}\{x,y\}\in S$ \cite{bouchet1987}. It turns out that
if all sets of $D$ have the same cardinality, then $D$ is a matroid
represented by its bases \cite{bouchet1987}. In this way, a delta-matroid can
be viewed as a generalization of the notion of matroid. A
delta-matroid $D$ is called \emph{even} if the cardinalities of the sets of
$D$ have equal parity. For $X\subseteq V$, we
define the \emph{twist} of $D$ by $X$ as the set system $D\ast X=(V,S^{\prime})$ with
$S^{\prime}=\{X\mathop{\mathrm{\Delta}}Y\mid Y\in S\}$. It turns out that
$D\ast X$ is an (even) delta-matroid if and only if $D$ is an (even) delta-matroid. 

\subsection{Representable and regular delta-matroids}

\label{ssec:repre_reg_dm}

For finite sets $X$ and $Y$, an $X\times Y$-matrix $A$ is a matrix where the
rows and columns of $A$ are indexed by $X$ and $Y$, respectively, and are not ordered. If
$W\subseteq X$ and $W\subseteq Y$, then $A[W]$ denotes the $W\times W$-matrix
obtained from $A$ by removing the entries outside $W$. We now fix a finite set
$V$. A $V\times V$-matrix $A$ over some field $\mathbb{F}$ is said to be
\emph{skew-symmetric} if $-A^{T}=A$ (note that we allow nonzero diagonal
entries in case $\mathbb{F}$ is of characteristic $2$). If $A$ is
skew-symmetric, then $\mathcal{D}_{A}=(V,S)$ with $S=\{X\mid
A[X]\mbox{ is nonsingular}\}$ is a delta-matroid \cite{bouchet1987}. For a
skew-symmetric $V\times V$-matrix $A$ over $\mathbb{F}$, $\mathcal{D}_{A}$ is
even if and only if all diagonal entries of $A$ are zero. A delta-matroid $D$
is said to be \emph{representable} over $\mathbb{F}$ if there is a
skew-symmetric $V\times V$-matrix $A$ over $\mathbb{F}$ such that
$D=\mathcal{D}_{A}\ast X$ for some $X\subseteq V$. This notion of
representability for delta-matroids generalizes the notion of representability
for matroids: a matroid $M$ is representable over $\mathbb{F}$ in the standard
matroid sense if and only if $M$ is representable over $\mathbb{F}$ in the
delta-matroid sense. Indeed, if $M$ is representable by a matrix
\[
\bordermatrix{
& X & Y \cr
X & I & E}
\]
in standard form, then one may verify that the delta-matroid $\mathcal{D}_{A}$
corresponding to skew-symmetric matrix $V\times V$-matrix
\[
A=\bordermatrix{
& X & Y \cr
X & 0 & E \cr
Y & -E^{\text{T}} & 0 \cr
}
\]
is such that $\mathcal{D}_{A}\ast X=M$. The converse also holds
\cite{bouchet1987}.

A $V\times V$-matrix $A$ over $\mathbb{R}$ is said to be \emph{principally
unimodular} if $\det(A[X])\in\{0,1,-1\}$ for all $X\subseteq V$. (In
particular, if $A$ is principally unimodular, then each entry of $A$ is equal
to $0$, $1$, or $-1$.) We say that $D$ is \emph{regular} if $D$ is
representable by a skew-symmetric principally unimodular matrix over
$\mathbb{R}$. The following result is shown by Geelen~\cite{GeelenPhD}.

\begin{theorem}
[\cite{GeelenPhD}]\label{thm:regular_dm_char} Let $D$ be an even
delta-matroid. Then the following statements are equivalent.

\begin{enumerate}
\item $D$ is regular,

\item $D$ is representable over every field, and

\item $D$ is representable over both $GF(2)$ and $GF(3)$.
\end{enumerate}
\end{theorem}

The notion of regularity for delta-matroids generalizes the notion of
regularity for matroids. Recall that a matroid $M$ is called \emph{regular} if
$M$ is representable by a totally unimodular matrix over $\mathbb{R}$, where a
matrix $B$ over $\mathbb{R}$ is said to be \emph{totally unimodular} if the
determinant of every submatrix of $B$ is equal to $0$, $1$, or $-1$. We may
assume that $B$ is in standard form $(I\quad E)$. Now, it is easy to verify
that
\[
\bordermatrix{
& X & Y \cr
X & I & E}
\]
is totally unimodular if and only if the skew-symmetric matrix $V\times
V$-matrix
\[
A=\bordermatrix{
& X & Y \cr
X & 0 & E \cr
Y & -E^{\text{T}} & 0 \cr
}
\]
with $V=X\cup Y$ is principally unimodular \cite{PU:skewsymm:1998}.

If $A$ is a skew-symmetric matrix over $GF(2)$ (equivalently, $A$ is symmetric
over $GF(2)$), then $\mathcal{D}_{A}$ uniquely determines $A$
\cite{Bouchet_1991_67}. Indeed, for $v \in V$, $A[\{v\}] = 1$ if and only if
$\{v\} \in\mathcal{D}_{A}$. Moreover, for $v, w \in V$ with $v \neq w$, we
have $A[\{v,w\}] = 1$ if and only if we have either $(\{v,w\} \in\mathcal{D}_{A})
\land ((\{v\} \notin\mathcal{D}_{A}) \lor(\{w\} \notin\mathcal{D}
_{A})$ or $(\{v,w\} \notin\mathcal{D}_{A})
\land (\{v\} \in\mathcal{D}_{A}) \land(\{w\} \in\mathcal{D}_{A})$. We say that $D$ is \emph{binary} if $D$ is representable over $GF(2)$.

\subsection{Eulerian delta-matroids}

A delta-matroid $D$ is said to be \emph{Eulerian} if $D = \mathcal{D}_{A(G)} *
X$ where $G$ is a circle graph and $X \subseteq V(G)$ \cite{GeelenPhD}. For
notational convenience, this definition is slightly different from
\cite{GeelenPhD} as there it is required that $X = \emptyset$. The next lemma
shows that this difference is not essential.

\begin{lemma}
\label{lem:circle_graph_iff_eulerian} Let $G$ be a simple graph. Then $G$ is a
circle graph if and only if $\mathcal{D}_{A(G)}$ is Eulerian.
\end{lemma}

\begin{proof}
The only if direction is trivial. To prove the converse, let $D =
\mathcal{D}_{A(G)}$ be Eulerian. Then $D * X = \mathcal{D}_{A(G^{\prime})}$
for some circle graph $G^{\prime}$. Hence, $\mathcal{D}_{A(G)} =
\mathcal{D}_{A(G^{\prime})}*X$. It is shown in \cite{bouchet1987} that this
implies that $G$ is locally equivalent to $G^{\prime}$. By
Theorems~\ref{theory} and \ref{bouchet}, $G$ is a circle graph as well.
\end{proof}

It follows from de~Fraysseix~\cite{F} that Eulerian delta-matroids are a
generalization of planar matroids (i.e., cycle matroids of planar graphs); see
also \cite[Theorem~4.16]{GeelenPhD}.

\begin{theorem}
[\cite{F}]\label{thm:eulerian_dm_gen_planar} Let $M$ be a matroid. Then $M$ is
planar if and only if $M$ is an Eulerian delta-matroid.
\end{theorem}

Since $\mathcal{D}_{A}$ uniquely determines $A$, a characterization of
Eulerian delta-matroids directly implies a characterization of circle graphs.
The following characterization of Eulerian delta-matroids is from
\cite{GeelenPhD}.

\begin{theorem}
[\cite{GeelenPhD}]\label{thm:eulerian_dm_char} Let $D$ be an even binary
delta-matroid, i.e., $D = \mathcal{D}_{A(G)} * X$ for some simple graph $G$
and $X \subseteq V(G)$. Then $D$ is Eulerian if and only if, for every graph
$G^{\prime}$ locally equivalent to $G$, $\mathcal{D}_{A(G^{\prime})}$ is regular.
\end{theorem}

In particular, every Eulerian delta-matroid is regular. The converse does not
hold: any regular matroid $M$ that is not planar is a counterexample by
Theorem~\ref{thm:eulerian_dm_gen_planar}. So, e.g., the cycle matroids of
$K_{3,3}$ and $K_{5}$ are regular, but they are not Eulerian delta-matroids.

Since $\mathcal{D}_{A(G)}$ uniquely determines $G$, it is natural to formulate
the notion of local equivalence for binary delta-matroids. For this we require
an additional operation on delta-matroids. For a delta-matroid $D$ over $V$
and $X\subseteq V$, \emph{loop complementation} of $D$ by $X$, denoted by
$D+X$, is defined by $Y\in D+X$ if and only if there are an odd number of
$Z\in D$ with $(Y-X)\subseteq Z \subseteq Y$. In general $D+X$ need not be a delta-matroid, but it turns out that the family of binary
delta-matroids is closed under loop complementation \cite[Section~5]{BH/PivotLoopCompl/09}.

We now use loop complementation to reformulate
Theorem~\ref{thm:eulerian_dm_char}.

\begin{theorem}
\label{thm:eulerian_dm_char_plus} Let $D$ be an even binary delta-matroid over
$V$. Then the following statements are equivalent.

\begin{enumerate}
\item $D$ is Eulerian,

\item for all $X \subseteq Y \subseteq V$ with $D * X + Y$ even, we have that
$D * X + Y$ is regular,

\item every even delta-matroid obtainable from $D$ by applying a sequence of
$+$ and $*$ operations is regular.
\end{enumerate}
\end{theorem}

\begin{proof}
Statement 3 implies statement 2 directly. For the converse, recall that it is shown in \cite[Theorem~12]{BH/PivotLoopCompl/09} that any delta-matroid
$D^{\prime}$ obtainable from $D$ by applying a sequence of $+$ and $*$ is of
the form $D * X + Y * Z$ with $X \subseteq Y \subseteq V$ and $Z \subseteq V$.
As twists preserve both evenness and regularity, $D * X + Y * Z$ is even (regular) if and only if $D * X + Y$ is
even (regular). Hence the last two statements are equivalent.

It is shown in \cite[Theorem~27]{BH/PivotLoopCompl/09} that for simple graphs $G$
and $G^{\prime}$, $G^{\prime}$ is locally equivalent to $G$ if and only if
$\mathcal{D}_{A(G^{\prime})}$ can be obtained from $\mathcal{D}_{A(G)}$ by
applying a sequence of $+$ and $*$ operations. By
Theorem~\ref{thm:eulerian_dm_char}, we obtain that statement 3 implies that
$D$ is Eulerian. Conversely, let $D$ be Eulerian. Then $D = \mathcal{D}_{A(G)}
* X$ for some simple graph $G$ and $X \subseteq V(G)$. Let $\varphi$ be a
sequence of $+$ and $*$ operations such that $D\varphi$ is even. Then
$(\mathcal{D}_{A(G)} * X) \varphi$ is even. Let $Y \in(\mathcal{D}_{A(G)} * X)
\varphi$. Then $(\mathcal{D}_{A(G)} * X) \varphi*Y$ contains the empty set and
so it is equal to $\mathcal{D}_{A(G^{\prime})}$ for some graph $G^{\prime}$ by
\cite[Proof of Theorem~8.2]{BH/NullityLoopcGraphsDM/10}. Hence $G^{\prime}$ is
locally equivalent to $G$ and so $\mathcal{D}_{A(G^{\prime})}$ is regular (as
$D$ is Eulerian). Thus $D\varphi= \mathcal{D}_{A(G^{\prime})}*Y$ is regular
and so statement 3 holds.
\end{proof}

\begin{corollary}
Let $M$ be a binary matroid. Then $M$ is planar if and only if every even
delta-matroid obtainable from $D$ by applying a sequence of $+$ and $*$
operations is regular.
\end{corollary}

\section{Characterizing circle graphs by multimatroid regularity}

\label{sec:char_circle_mm_regularity}

In this section we reformulate Theorem~\ref{thm:eulerian_dm_char} in terms of
isotropic matroids.

\subsection{Sheltering matroids}

First we recall some notions and notation for the theory of multimatroids
\cite{B1}. Let $\Omega$ be a partition of a finite set $U$. A set $T \subseteq U$
is called a \emph{transversal} (\emph{subtransversal}, respectively) of $\Omega$ if
$|T \cap\omega|=1$ ($|T \cap\omega|\leq1$, respectively) for all $\omega\in\Omega$.
We denote the set of transversals of $\Omega$ by $\mathcal{T}(\Omega)$ and the
set of subtransversals of $\Omega$ by $\mathcal{S}(\Omega)$. A $p \subseteq U$
is called a \emph{skew pair} of $\omega\in\Omega$ if $|p|=2$ and $p
\subseteq\omega$. We say that $\Omega$ is a \emph{$q$-partition} if $q =
|\omega|$ for all $\omega\in\Omega$.

Multimatroids form a generalization of matroids. Like matroids, multimatroids
can be defined in terms of rank, circuits, independent sets, etc. Here they
are defined in terms of independent sets.

\begin{definition}
[\cite{B1}]\label{def:multimatroid} Let $\Omega$ be a partition of a finite
set of $U$. A \emph{multimatroid} $Z$ over $(U,\Omega)$, described by its
independent sets, is a triple $(U,\Omega,\mathcal{I})$, where $\mathcal{I}%
\subseteq\mathcal{S}(\Omega)$ is such that:

\begin{enumerate}
\item for each $T \in\mathcal{T}(\Omega)$, $(T,\mathcal{I}\cap2^{T})$ is a
matroid (described by its independent sets) and

\item for any $I \in\mathcal{I}$ and any skew pair $p = \{x,y\}$ of some
$\omega\in\Omega$ with $\omega\cap I = \emptyset$, $I \cup\{x\} \in
\mathcal{I}$ or $I \cup\{y\} \in\mathcal{I}$.
\end{enumerate}
\end{definition}

If $\Omega$ is a $q$-partition, then we say that $Q$ is a \emph{$q$-matroid}. Note that a 1-matroid is essentially a matroid because in this case the partition of $\Omega$ into singletons does not capture any additional information.
A \emph{basis} of $Z$ is a set in $\mathcal{I}$ maximal with respect to
inclusion. For $X\subseteq U$, we define $Z-X=(Z-X,\Omega^{\prime}%
,\mathcal{I})$ with $\Omega^{\prime}=\{\omega\setminus X\mid\omega\in\Omega\}$
and $\mathcal{I}^{\prime}=\{I\in\mathcal{I}\mid I\cap X=\emptyset\}$. A
multimatroid $Z$ is called \emph{nondegenerate} if $|\omega|>1$ for all
$\omega\in\Omega$. Also, nondegenerate multimatroid $Z$ is called \emph{tight} if for every independent $S\in\mathcal{S}(\Omega)$ with $|S|=|\Omega|-1$, there is an $x\in U\backslash
S$ such that $S\cup\{x\}$ is a dependent subtransversal.

We now consider the related notion of sheltering matroid introduced in
\cite{BT1}. It is a generalization of the matroid $M[IAS(G)]$.

\begin{definition}
A \emph{sheltering matroid} is a tuple $Q = (M,\Omega)$ where $M$ is a matroid
over some ground set $U$ and $\Omega$ is a partition of $U$ such that for any
independent set $I \in\mathcal{S}(\Omega)$ of $M$ and for any skew pair $p =
\{x,y\}$ of $\omega\in\Omega$ with $\omega\cap I = \emptyset$, $I \cup\{x\}$
or $I \cup\{y\}$ is an independent set of $M$.
\end{definition}

It is shown in \cite{BT1} that $(M[IAS(G)],\Omega)$ is a sheltering matroid
with $\Omega$ the set of vertex triples of $G$.

Matroid notions carry over straightforwardly to sheltering matroids. For
example, for $X \subseteq U$, we define the \emph{deletion} of $X$ from $Q$ by
$Q - X = (M - X,\Omega^{\prime})$ with $\Omega^{\prime}= \{ \omega\setminus X
\mid\omega\in\Omega\}$.

Note that if $Q=(M,\Omega)$ is a sheltering matroid, then $\mathcal{Z}%
({Q})=(U,\Omega,\mathcal{I})$ with $U$ the ground set of $M$ and
$\mathcal{I}=\{I\in\mathcal{S}(\Omega)\mid
I\mbox{ is an independent set of }M\}$ is a multimatroid. We say that
$\mathcal{Z}({Q})$ is the \emph{multimatroid corresponding} to $Q$. Also, we
say that $M$ \emph{shelters} the multimatroid $\mathcal{Z}({Q})$. Not every
multimatroid is sheltered by a matroid \cite{B1}. If $\mathcal{Z}({Q})$ is a
$q$-matroid, then $Q$ is called a \emph{$q$-sheltering matroid}. Moreover, $Q$
is called \emph{tight} when $\mathcal{Z}({Q})$ is tight.

Let $Q_{1} = (M_{1},\Omega_{1})$ and $Q_{2} = (M_{2},\Omega_{2})$ be
sheltering matroids. An \emph{isomorphism} $\varphi$ from $Q_{1}$ to $Q_{2}$
is an isomorphism from $M_{1}$ to $M_{2}$ that respects the skew classes,
i.e., if $x$ and $y$ are elements of the ground set of $M_{1}$, then $x$ and
$y$ are in a common skew class of $\Omega_{1}$ if and only if $\varphi(x)$ and
$\varphi(y)$ are in a common skew class of $\Omega_{2}$.

For a $q$-partition $\Omega$ (of some finite set $U$), a \emph{transversal
$q$-tuple} of $\Omega$ is a sequence $\tau=(T_{1},\ldots,T_{q})$ of $q$
mutually disjoint transversals of $Q$. Note that the transversals of $\tau$ are ordered.

\subsection{Strongly Representable and Regular 2-Sheltering Matroids}

\emph{Note:} for notational convenience, from now on we often assume a given
fixed $2$-partition $\Omega$ of some finite set $U$ and a transversal 2-tuple
$\tau=(T_{1},T_{2})$ of $\Omega$.

We say that a matrix $B$ over some field $\mathbb{F}$ \emph{strongly represents} the
$2$-sheltering matroid $Q=(M,\Omega)$ if
\[
B=\bordermatrix{
& T_1 & T_2 \cr
& I & A
}
\]
is a standard representation of $M$ with $A$ a skew-symmetric matrix for some
transversal $2$-tuple $\tau=(T_{1},T_{2})$. We denote $Q$ by $\mathcal{Q}%
(A,\tau,2)$. We use the word ``strongly'' because in the companion paper \cite{BT1}
we also consider weaker notions of representability. Also, a $2$-sheltering matroid $Q$ is said to be
\emph{strongly representable} over $\mathbb{F}$ if there is such a matrix $B=
\begin{pmatrix}
I & A
\end{pmatrix}
$ and
transversal $2$-tuple $\tau$ such that $Q=\mathcal{Q}(A,\tau,2)$.

We say that a 2-matroid $Z$ is \emph{strongly representable} over $\mathbb{F}$ if there
is a $2$-sheltering matroid $Q$ strongly representable over $\mathbb{F}$ such that
$\mathcal{Z}({Q}) = Z$. We say that $Q$ ($Z$, respectively) is \emph{strongly binary} if $Q$
($Z$, respectively) is strongly representable over $GF(2)$.

We now define a notion of regularity for $2$-sheltering matroids. We say that
a $2$-sheltering matroid $Q=(M,\Omega)$ is \emph{t-regular} if $Q$ has a
strong representation $B=%
\begin{pmatrix}
I & A
\end{pmatrix}
$ over $\mathbb{R}$ such that for each
$T\in\mathcal{T}(\Omega)$, the determinant of the matrix obtained from $B$ by
restricting to the columns of $T$ is equal to $0$, $1$, or $-1$. Similarly, we
say that a 2-matroid $Z$ is \emph{t-regular} if there is a t-regular
$2$-sheltering matroid $Q$ such that $\mathcal{Z}({Q})=Z$.

Note that t-regularity for a $2$-sheltering matroid $(M,\Omega)$ is not the
same as regularity of $M$. On the one hand, it can happen that $M$ is regular
but $(M,\Omega)$ is not t-regular, if no representation of $M$ has the
required skew symmetry with respect to $\Omega$. See \cite[Section 2.3]{BT1}
for an example. On the other hand, it can happen that $(M,\Omega)$ is
t-regular but $M$ is not regular. For example, let $M$ be represented by the
matrix
\[
B=\bordermatrix{
& t_{1,1} & t_{1,2} & t_{1,3} & t_{1,4} & t_{2,1} & t_{2,2} & t_{2,3} & t_{2,4} \cr
& 1 & 0 & 0 & 0 & 0 & 1 & 1 & 1 \cr
& 0 & 1 & 0 & 0 & -1 & 0 & 1 & 1 \cr
& 0 & 0 & 1 & 0 & -1 & -1 & 0 & 1 \cr
& 0 & 0 & 0 & 1 & -1 & -1 & -1 & 0
}
\]
and let $\Omega=\{\omega_{1},\omega_{2},\omega_{3},\omega_{4}\}$ where
$\omega_{j}=\{t_{1,j},t_{2,j}\}$. Then $B$ satisfies the definition of
t-regularity with respect to $\Omega$. The minor $M/\{t_{1,1},t_{1,3}\}$ is
represented by the matrix obtained from $B$ by removing the first and third
rows and columns:%
\[
\bordermatrix{
&  t_{1,2} & t_{1,4} & t_{2,1} & t_{2,2} & t_{2,3} & t_{2,4} \cr
& 1 & 0 & -1 & 0 & 1 & 1 \cr
& 0 & 1 & -1 & -1 & -1 & 0
}.
\]
Evidently $(M/\{t_{1,1},t_{1,3}\})-t_{2,2}-t_{2,4}$ is isomorphic to $U_{2,4}%
$, so $M$ is not regular.

We now make the following observation.

\begin{lemma}
\label{lem:pu_det_associated} Let $A$ be a $V\times V$-matrix over $\mathbb{R}$.
Then $A$ is principally unimodular if and only if, for each transversal $T$ of
$\Omega$, the determinant of the matrix obtained from
\[
\bordermatrix{
& \bar{V} & V \cr
& I & A
}
\]
by restricting to the columns of $T$ is equal to $0$, $1$, or $-1$.
\end{lemma}

\subsection{Eulerian and t-regular 3-Sheltering Matroids}

\emph{Note:} for notational convenience, from now on we often assume a given
fixed $3$-partition $\Omega$ of some finite set $U$ and a transversal
$3$-tuple $\tau=(T_{1},T_{2},T_{3})$ of $\Omega$.

We observed that $(M[IAS(G)],\Omega)$ is a $3$-sheltering matroid. The following straightforward reformulation of \cite[Theorem~41]{Tnewnew} shows that $(M[IAS(G)],\Omega)$ is tight.

\begin{lemma}
[Theorem~41 of \cite{Tnewnew}]\label{lem:ias_tight_3matroid} Let $A$ be a
$V\times V$-symmetric matrix over $GF(2)$ and $|\Omega|=|V|$. Let $M$ be the
column matroid of
\[
B=\bordermatrix{
& T_1 & T_2 & T_3\cr
& I & A & A+I
}.
\]
Then $(M,\Omega)$ is a tight $3$-sheltering matroid.
\end{lemma}

We denote the tight $3$-sheltering matroid $(M,\Omega)$ of
Lemma~\ref{lem:ias_tight_3matroid} by $\mathcal{Q}(A,\tau,3)$. We say that a
$3$-sheltering matroid $Q$ is \emph{isotropic} if $Q = \mathcal{Q}(A,\tau,3)$
for some symmetric matrix $A$ over $GF(2)$. Note that each isotropic
$3$-sheltering matroid is isomorphic to the isotropic $3$-sheltering matroid
$(M[IAS(G)],\Omega)$ for some graph $G$. Also note that $\mathcal{Q}%
(A,\tau,3)-T_{3} = \mathcal{Q}(A,(T_{1},T_{2}),2)$. We say that a $3$-matroid
$Z$ is \emph{isotropic} if $Z = \mathcal{Z}(Q)$ for some isotropic
$3$-sheltering matroid $Q$.

For convenience, we also write $\mathcal{Q}(G,\tau,3)$ ($\mathcal{Q}%
(G,\tau^{\prime},2)$, respectively) to denote $\mathcal{Q}(A(G),$ $\tau,3)$
($\mathcal{Q}(A(G),\tau^{\prime},2)$, respectively) for simple graphs $G$. For $i\in\{2,3\}$,
an $i$-sheltering matroid $Q$ is called \emph{Eulerian} if $Q=\mathcal{Q}%
(G,\tau,i)$ for some circle graph $G$ and transversal $i$-tuple $\tau$.

We now define the notion of t-regularity for $3$-sheltering matroids.

\begin{definition}
Let $Q = (M,\Omega)$ be a $3$-sheltering matroid. Then $Q$ is called
\emph{t-regular} if, for every $T \in\mathcal{T}(\Omega)$, $Q - T$ is
t-regular whenever $Q-T$ is tight.
\end{definition}

The main result of this section is as follows and is proved in the next subsection.

\begin{theorem}
\label{thm:eulerian_dm_char_ias} Let $Q$ be an isotropic $3$-sheltering
matroid. Then $Q$ is Eulerian if and only if $Q$ is t-regular.
\end{theorem}

\subsection{Proof of Theorem~\ref{thm:eulerian_dm_char_ias}}

To prove Theorem~\ref{thm:eulerian_dm_char_ias} we translate
Theorem~\ref{thm:eulerian_dm_char_plus} from delta-matroids to sheltering matroids.

We recall the following results from \cite{B1}.

\begin{lemma}
[\cite{B1}]\label{lem:2m_dm_section} Let $\tau= (T_{1},T_{2})$ be a
transversal $2$-tuple of $\Omega$. If $D = (V,S)$ is a delta-matroid with $|V|
= |\Omega|$, then $\mathcal{Z}(D,\tau,2) := (U,\Omega,\mathcal{B})$ with
$\mathcal{B} = \{ X \in\mathcal{T}(\Omega) \mid\pi(X \cap T_{2}) \in D\}$ is a 2-matroid.
\end{lemma}

\begin{lemma}
[\cite{B1}]\label{lem:mm_inverse_dm} Let $\tau$ be a transversal 2-tuple of
$\Omega$. The mapping $D \mapsto\mathcal{Z}(D,\tau,2)$ is a one-to-one
correspondence from the family of delta-matroids $D$ over $V$ to the family of
2-matroids over $(U,\Omega)$.
\end{lemma}

Moreover, it is shown in \cite{B3} that $D$ is even if and only if
$\mathcal{Z}(D,\tau,2)$ is tight.  By Lemma~\ref{lem:pu_det_associated}, we have the following.
\begin{lemma}
\label{lem:regular_sm_dm} For every binary delta-matroid $D$, the $2$-matroid
$\mathcal{Z}(D,\tau,2)$ is t-regular if and only if $D$ is regular.
\end{lemma}

It is easy to verify that for every symmetric matrix $A$ over $GF(2)$ we have
$\mathcal{Z}(\mathcal{Q}(A,\tau,2))$ $ = \mathcal{Z}(\mathcal{D}_{A},\tau,2)$,
see also \cite{BT1}.

We now turn to $3$-matroids.

\begin{lemma}
[Lemma 20 and Theorem 22 of \cite{BT1}]\label{lem:unique_reconstruct_matrix} Let $Z$ be an isotropic
$3$-matroid. Then there is a unique isotropic $3$-sheltering matroid $Q$ such that $\mathcal{Z}({Q}) = Z$.
\end{lemma}

We say that a 3-matroid $Z$ is \emph{t-regular} if there is a t-regular
$3$-sheltering matroid $Q$ such that $\mathcal{Z}({Q}) = Z$.

\begin{lemma}
[Theorems~13 and 16 of \cite{BH3}]\label{lem:unique_3m} Let $\Omega$ be a
$3$-partition of some finite set $U$ and let $T \in\mathcal{T}(\Omega)$. If
$Z$ is a strongly binary 2-matroid over $(U \setminus T,\Omega^{\prime})$ with
$\Omega^{\prime}= \{ \omega\setminus T \mid\omega\in\Omega\}$, then there is a
unique tight 3-matroid $Z'$ over $(U,\Omega)$ with $Z'-T=Z$.
\end{lemma}

Let $D$ be a binary delta-matroid, then the unique tight 3-matroid
corresponding to $\mathcal{Z}(D,(T_{1},T_{2}),2)$ (with respect to $\tau$ and
$\Omega$), cf. Lemma~\ref{lem:unique_3m}, is denoted by $\mathcal{Z}%
(D,\tau,3)$. Lemma~\ref{lem:unique_3m} implies the following lemma.

A \emph{projection} $\pi$ of $\Omega$
is a function $U\rightarrow V$ with $|V|=|\Omega|$ such that $\pi(x)=\pi(y)$
if and only if $x,y\in\omega$ for some $\omega\in\Omega$.
Let $\tau= (T_{1},T_{2},T_{3})$ be a transversal 3-tuple and $\pi: U \to V$ be a projection. For $v \in
V$, let $p_{i}$ ($i \in\{1,2,3\}$) be the unique skew pair of $\omega=
\pi^{-1}(v)$ with $p_{i} \cap T_{i} = \emptyset$. We define $\tau*v = (T_{1}
\mathop{\mathrm{\Delta}} p_{3}, T_{2} \mathop{\mathrm{\Delta}} p_{3}, T_{3})$
and $\tau+v = (T_{1}, T_{2} \mathop{\mathrm{\Delta}} p_{1}, T_{3}
\mathop{\mathrm{\Delta}} p_{1})$

\begin{lemma}
[Lemma~15 of \cite{BH3}]\label{lem:dm_ops_mm_label_cng} Fix a $3$-partition
$\Omega$ of some finite set $U$, a transversal $3$-tuple $\tau$ of $\Omega$,
and a projection $\pi: U \to V$. Let $D$ be a binary delta-matroid over $V$.
Then $\mathcal{Z}(D,\tau,3) = \mathcal{Z}(D+v,\tau+v,3) = \mathcal{Z}%
(D*v,\tau*v,3)$.
\end{lemma}

\begin{lemma}
\label{lem:3m_mat_dm} Let $A$ be a $V \times V$-symmetric matrix over $GF(2)$.
Then $\mathcal{Z}(\mathcal{Q}(A,\tau,3)) = \mathcal{Z}(\mathcal{D}_{A}%
,\tau,3)$.
\end{lemma}

\begin{proof}
It is easy to verify that $\mathcal{Z}(\mathcal{Q}(A,(T_{1},T_{2}),2)) =
\mathcal{Z}(\mathcal{D}_{A},(T_{1},T_{2}),2)$ and the extension to transversal
$3$-tuples follows from Lemma~\ref{lem:unique_3m}.
\end{proof}

\begin{lemma}
\label{lem:dm_eul_mm_treg} Let $D$ be a binary delta-matroid. Then $D$ is
Eulerian if and only if $\mathcal{Z}(D,\tau,3)$ is t-regular (again we assume
an arbitrary fixed transversal $3$-tuple $\tau$).
\end{lemma}

\begin{proof}
By Theorem~\ref{thm:eulerian_dm_char_plus}, $D$ is Eulerian if and only if
every even delta-matroid obtainable from $D$ by applying a sequence of $+$ and
$*$ operations is regular.

First assume that $\mathcal{Z}(D,\tau,3)$ is t-regular. Let $\varphi$ be a
sequence of $+$ and $*$ operations such that $D\varphi$ is even. We have
$\mathcal{Z}(D,\tau,3) = \mathcal{Z}(D\varphi,\tau\varphi,3)$. Let
$\tau\varphi= (T_{1}^{\prime},T_{2}^{\prime},T_{3}^{\prime})$. Then
$\mathcal{Z}(D,\tau,3) - T_{3}^{\prime}= \mathcal{Z}(D\varphi,\tau\varphi,3) -
T_{3}^{\prime}= \mathcal{Z}(D\varphi,(T_{1}^{\prime},T_{2}^{\prime}),2)$ is
tight because $D\varphi$ is even. Thus $\mathcal{Z}(D\varphi,(T_{1}^{\prime
},T_{2}^{\prime}),2) = \mathcal{Z}(D,\tau,3) - T_{3}^{\prime}$ is t-regular
and by Lemma~\ref{lem:regular_sm_dm}, $D\varphi$ is regular. Consequently, $D$
is Eulerian.

The reverse implication is similar. Thereto, assume that $D$ is Eulerian. Let
$T \in\mathcal{T}(\Omega)$ such that $\mathcal{Z}(D,\tau,3)-T$ is tight. By
Lemma~\ref{lem:dm_ops_mm_label_cng} there is a sequence $\varphi$ of $+$ and
$*$ operations such that $\mathcal{Z}(D,\tau,3) = \mathcal{Z}(D\varphi
,\tau\varphi,3)$ and $\tau\varphi= (T_{1}^{\prime},T_{2}^{\prime},T)$. Now,
$\mathcal{Z}(D,\tau,3)-T = \mathcal{Z}(D\varphi,\tau\varphi,3) - T =
\mathcal{Z}(D\varphi,\tau^{\prime},2)$ where $\tau^{\prime}= (T_{1}^{\prime
},T_{2}^{\prime})$. Since $\mathcal{Z}(D,\tau,3)-T = \mathcal{Z}(D\varphi
,\tau^{\prime},2)$ is tight, $D\varphi$ is even. Since $D$ is Eulerian, we
have that $D\varphi$ is regular. By Lemma~\ref{lem:regular_sm_dm},
$\mathcal{Z}(D\varphi,\tau^{\prime},2) = \mathcal{Z}(D,\tau,3)-T$ is t-regular.
\end{proof}


We are now ready to prove Theorem~\ref{thm:eulerian_dm_char_ias}.

\begin{proof}
\emph{(of Theorem~\ref{thm:eulerian_dm_char_ias})} Let $\tau= (T_{1}%
,T_{2},T_{3})$ and let $Q = \mathcal{Q}(A,\tau,3)$ for some symmetric matrix
$A$ over $GF(2)$. By Lemma~\ref{lem:3m_mat_dm}, $\mathcal{Z}(Q) =
\mathcal{Z}(\mathcal{D}_{A},\tau,3)$.

Assume first that $Q = \mathcal{Q}(A,\tau,3)$ is t-regular. Then
$\mathcal{Z}(\mathcal{D}_{A},\tau,3) = \mathcal{Z}(Q)$ is t-regular. By
Lemma~\ref{lem:dm_eul_mm_treg}, $\mathcal{D}_{A}$ is Eulerian. Hence
$\mathcal{D}_{A} = \mathcal{D}_{A(G)}*X$ for some circle graph $G$ and $X
\subseteq V(G)$. By Lemma~\ref{lem:dm_ops_mm_label_cng}, $\mathcal{Z}%
(\mathcal{D}_{A(G)}*X,\tau,3) = \mathcal{Z}(\mathcal{D}_{A(G)},\tau*X,3)$. By
Lemma~\ref{lem:3m_mat_dm}, $\mathcal{Z}(\mathcal{D}_{A(G)},\tau*X,3) =
\mathcal{Z}(A(G),\tau*X,3)$. Consequently, $\mathcal{Z}({Q}) = \mathcal{Z}%
(\mathcal{Q}(A(G),\tau*X,3))$ for some circle graph $G$. By
Lemma~\ref{lem:unique_reconstruct_matrix}, $Q = \mathcal{Q}(A(G),\tau*X,3)$
and so $Q$ is Eulerian.

Now assume that $Q$ is Eulerian. Then $\mathcal{Z}({Q}) = \mathcal{Z}%
(\mathcal{Q}(A(G),\tau,3))$ for some circle graph $G$. Hence $\mathcal{D}%
_{A(G)}$ is Eulerian and by Lemma~\ref{lem:dm_eul_mm_treg} $\mathcal{Z}%
(\mathcal{D}_{A(G)},\tau,3)$ is t-regular. By Lemma~\ref{lem:3m_mat_dm},
$\mathcal{Z}(\mathcal{Q}(A(G),\tau,3)) = \mathcal{Z}({Q})$ is t-regular. By
Lemma~\ref{lem:unique_reconstruct_matrix}, $Q = \mathcal{Q}(A(G),\tau,3)$ is t-regular.
\end{proof}

\section{4-regular graphs}
\label{sec:4reg_graphs}

In this section we discuss the relationship between interlacement graphs and
circuits of 4-regular graphs.

We begin by establishing some notation and terminology. We think of an edge of
a graph as consisting of two distinct \emph{half-edges}, one incident on each
end-vertex of the original edge. A circuit of length $\ell$ is then a sequence
$v_{1},h_{1},h_{1}^{\prime},v_{2},h_{2},\ldots,h_{\ell}^{\prime},v_{\ell+1}%
=v_{1}$ of vertices and half-edges, with the property that for each $i$,
$h_{i-1}^{\prime}$ and $h_{i}$ are half-edges incident on $v_{i}$ and
$\{h_{i},h_{i}^{\prime}\}$ is an edge incident on $v_{i}$ and $v_{i+1}$.
Vertices may appear repeatedly on a circuit, but there must be $2\ell$
different half-edges. We do not distinguish between circuits that differ only
in orientation or starting point. That is, if $v_{1},h_{1},h_{1}^{\prime
},v_{2},h_{2},\ldots,h_{\ell}^{\prime},v_{\ell+1}=v_{1}$ is a circuit then for
each index $i$, the same circuit is represented by $v_{i},h_{i},\ldots,h_{\ell
}^{\prime},v_{\ell+1}=v_{1},h_{1},\ldots,h_{i-1}^{\prime},v_{i}$ and
$v_{i},h_{i-1}^{\prime},\ldots,h_{1},v_{1}=v_{\ell+1},h_{\ell}^{\prime}%
,\ldots,h_{i},v_{i}$.

\subsection{The transition matroid}

Notice that there are two ways to write a circuit as a sequence of pairs of
half-edges: one way is to pair consecutive half-edges of an edge, and the
other is to pair consecutive half-edges incident at a vertex. We call each of
these latter pairs $\{h_{i}^{\prime},h_{i+1}\}$ a \emph{single transition}. A
\emph{transition} at a vertex of a 4-regular graph is a pair of disjoint
single transitions at the same vertex; the $3\left\vert V(F)\right\vert
$-element set that contains all the transitions of $F$ is denoted
$\mathfrak{T}(F)$. (Our terminology is slightly nonstandard; for instance
Bouchet used \textquotedblleft transition\textquotedblright\ and
\textquotedblleft bitransition\textquotedblright\ rather than
\textquotedblleft single transition\textquotedblright\ and \textquotedblleft
transition\textquotedblright.)

An Euler system $C$ of $F$ gives rise to a notational scheme for
$\mathfrak{T}(F)$. First, arbitrarily choose an orientation for each circuit
of $C$. Then the three transitions at a vertex\ $v$ of $F$ may be described as
follows: one is used by $C$, one is not used by $C$ and is consistent with the
orientation of the circuit of $C$ incident at $v$, and the last one is
inconsistent with this orientation. (It is easy to see that changing the
orientations of some circuits of $C$ will not change the description of any
transition.) We label these three transitions $\phi_{C}(v)$, $\chi_{C}(v)$ and
$\psi_{C}(v)$ respectively.

It is important to note that different Euler systems of $F$ give rise to
different notational schemes for $\mathfrak{T}(F)$. That is, a particular
transition $\tau\in\mathfrak{T}(F)$ may be labeled $\phi$ for some Euler
systems, $\chi$ for others, and $\psi$ for the rest. For example, the reader
might take a moment to verify that if $C$ and $D$ are the Euler circuits of
$K_{5}$ indicated on the left and right in~Figure \ref{circmf16}
(respectively), and $v$ is the vertex at the bottom of the figure, then
$\phi_{C}(v)=\psi_{D}(v)$, $\chi_{C}(v)=\phi_{D}(v)$ and $\psi_{C}(v)=\chi
_{D}(v)$.%

\begin{figure}
[ptb]
\begin{center}
\includegraphics[
trim=2.271195in 8.024183in 1.340795in 1.070845in,
height=1.4607in,
width=3.6919in
]%
{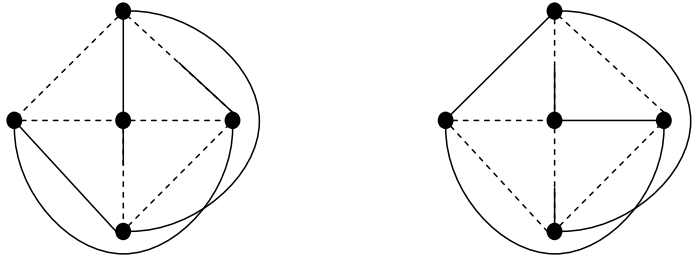}%
\caption{Two Euler circuits of $K_{5}$.}%
\label{circmf16}%
\end{center}
\end{figure}

Given an Euler system $C$ of a 4-regular graph $F$, we use the notational
scheme for $\mathfrak{T}(F)$ to associate transitions of $F$ with elements of
the isotropic matroid of the interlacement graph $\mathcal{I}(C)$. That is, if%
\[
IAS(\mathcal{I}(C))=%
\begin{pmatrix}
I & A(\mathcal{I}(C)) & I+A(\mathcal{I}(C))
\end{pmatrix}
\text{, }%
\]
where $I$ is an identity matrix and $A(\mathcal{I}(C))$ is the adjacency
matrix of $\mathcal{I}(C)$, then the $v$ column of $I$ is associated with the
transition $\phi_{C}(v)$, the $v$ column of $A(\mathcal{I}(C))$ is associated
with the transition $\chi_{C}(v)$, and the $v$ column of $I+A(\mathcal{I}(C))$
is associated with the transition $\psi_{C}(v)$.

This may to seem to define many different matroids on $\mathfrak{T}(F)$, but
in fact there is only one:

\begin{theorem}
(\cite{Tra}) Let $C$ and $D$ be any two Euler systems of a 4-regular graph $F$.
Then the matrices $IAS(\mathcal{I}(C))$ and $IAS(\mathcal{I}(D))$ represent
the same binary matroid on $\mathfrak{T}(F)$.
\end{theorem}

We call the matroid on $\mathfrak{T}(F)$ defined by any matrix
$IAS(\mathcal{I}(C))$ the \emph{transition matroid} of $F$, and denote it
$M_{\tau}(F)$. We refer to the transverse circuits, transverse matroids and
vertex triples of $\mathcal{I}(C)$ as transverse circuits, transverse matroids
and vertex triples of $M_{\tau}(F)$.

\subsection{Circuit partitions and touch-graphs}

As mentioned in Theorem \ref{jaeger2}, all transverse matroids of circle
graphs are cographic. The first special case of this property was observed by
Jaeger \cite{J}, who proved that if $G$ is a circle graph then the binary
matroid represented by the adjacency matrix of $G$ -- that is, the transverse
matroid consisting of all the $\chi$ elements of $W(G)$ -- is cographic.
Jaeger explained this result further in \cite{J1}, using the core vectors of
walks in a graph drawn on a surface. In his papers introducing isotropic
systems, Bouchet \cite{Bi1, Bi2} gave a different definition, related to
Jaeger's core vectors.

\begin{definition}
\label{touch}Let $F$ be a 4-regular graph, and $P$ a partition of the edge-set
of $F$ into edge-disjoint circuits. The \emph{touch-graph} $Tch(P)$ is the
graph with a vertex for each circuit of $P$ and an edge for each vertex of
$F$, the edge corresponding to $v$ incident on the vertex or vertices
corresponding to circuits of $P$ that are incident at $v$.
\end{definition}

Definition \ref{touch} is important for us because touch-graphs of circuit
partitions of a 4-regular graph $F$ are closely related to transverse matroids
of $M_{\tau}(F)$. This connection is given by the following result of
\cite{Tnew}. (See also \cite{BHT} and \cite{T5} for related results.) Recall that the \emph{vertex cocycle} associated with a vertex in a graph is the set of non-loop edges incident on that vertex.

\begin{theorem}
\label{coreker}(\cite{Tnew}) Let $P$ be a circuit partition of a 4-regular graph
$F$, let $C$ be an Euler system of $F$, and let $M(C,P)$ be the submatrix of
$IAS(\mathcal{I}(C))$ that includes those columns that correspond to
transitions included in $P$. Then the vertex cocycles of $Tch(P)$ span the
vector space%
\[
\ker M(C,P)=\{w\in GF(2)^{V(F)}\mid M(C,P)\cdot w=0\}\text{.}%
\]

\end{theorem}

We state the following direct consequence explicitly, for ease of reference.

\begin{corollary}
\label{touchmat}
If $G$ is a circle graph then the transverse matroids of $G$ are all cographic.
\end{corollary}

\begin{proof}
Let $C$ be an Euler system of a 4-regular graph $F$, with $G=\mathcal{I}(C)$. Theorem~\ref{coreker} tells us that for every circuit partition $P$ of $F$, the binary matroid represented by $M(C,P)$ is the cocycle matroid of $Tch(P)$. The corollary follows because the matroids represented by the various $M(C,P)$ matrices are the transverse matroids of $G$.
\end{proof}

Theorem \ref{coreker} generalizes and unifies ideas of Bouchet \cite{Bi1, B4}
and Jaeger \cite{J, J1}. Jaeger's core vector theory includes a special case
of Theorem \ref{coreker}, which requires (in our notation) that $P$ not follow
the $\phi_{C}(v)$ transition at any vertex. Bouchet's discussion of graphic
isotropic systems includes the cocycle spaces of touch-graphs, but does not
include an explicit matrix formulation.%

\begin{figure}
[ptb]
\begin{center}
\includegraphics[
trim=2.543092in 8.032988in 2.276293in 1.080750in,
height=1.4451in,
width=2.7856in
]%
{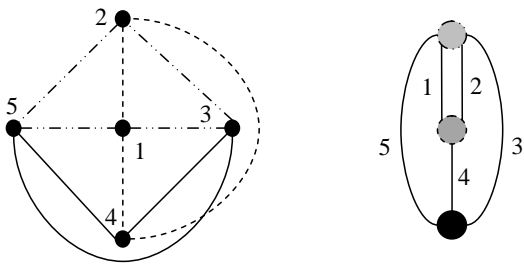}%
\caption{A circuit partition in $K_{5}$, and its touch-graph.}%
\label{circmf5}%
\end{center}
\end{figure}

As an example, consider the circuit partition $P$ of $K_{5}$ illustrated in
Figure \ref{circmf5}. The reader may verify that if $C$ and $D$ are the Euler
circuits of $K_{5}$ indicated in Figure \ref{circmf16} then the transitions
appearing in $P$ are $\phi_{C}(1)=\psi_{D}(1)$, $\phi_{C}(2)=\psi_{D}(2)$,
$\chi_{C}(3)=\chi_{D}(3)$, $\chi_{C}(4)=\phi_{D}(4)$ and $\phi_{C}(5)=\chi
_{D}(5)$. Consequently
\[
M(C,P)=%
\begin{pmatrix}
1 & 0 & 1 & 1 & 0\\
0 & 1 & 1 & 1 & 0\\
0 & 0 & 0 & 0 & 0\\
0 & 0 & 0 & 0 & 0\\
0 & 0 & 1 & 0 & 1
\end{pmatrix}
\text{ and }M(D,P)=%
\begin{pmatrix}
1 & 1 & 1 & 0 & 1\\
1 & 1 & 1 & 0 & 1\\
1 & 1 & 0 & 0 & 0\\
1 & 0 & 0 & 1 & 1\\
1 & 1 & 0 & 0 & 0
\end{pmatrix}
\text{.}%
\]
It is a simple matter to verify that the nonzero elements of $\ker M(C,P)=\ker
M(D,P)$ are the vertex cocycles of $Tch(P)$, i.e., the column vectors obtained
by transposing $(1,1,1,0,1)$, $(1,1,0,1,0)$ and $(0,0,1,1,1)$.

The following consequence of Theorem \ref{coreker} allows us to study the
transverse circuits of circle graphs using circuits of 4-regular graphs.

\begin{corollary}
\label{corekercor}Let $\gamma$ be a circuit of a 4-regular graph $F$, which is
not an Euler circuit of a connected component. Let $\tau(\gamma)$ denote the
set of transitions involved in $\gamma$. Then $\tau(\gamma)$ is a dependent
set of $M_{\tau}(F)$.
\end{corollary}

\begin{proof}
Let $\gamma$ be $v_{1},h_{1},h_{1}^{\prime},v_{2},h_{2},\ldots,h_{\ell}^{\prime
},v_{\ell+1}=v_{1}$. Then $V(F)=V_{0}\cup V_{1}\cup V_{2}$, where
$V_{k}=\{v\in V(F)\mid$ there are precisely $k$ values of $i\in\{1,\ldots,\ell\}$
with $v_{i}=v\}$. If $V_{1}=\emptyset$ then $\gamma$ includes all four
half-edges incident on each of $v_{1},\ldots,v_{\ell}$. Consequently no vertex of
$V_{2}$ neighbors any vertex outside $V_{2}$, so $\gamma$ is an Euler circuit
of a connected component with vertex-set $V_{2}$. By hypothesis, this is not
the case; hence $V_{1}\not =\emptyset$.

Arbitrarily choose transitions at vertices of $V_{0}$, and let $P$ denote the
circuit partition that involves these transitions in addition to the elements
of $\tau(\gamma)$. Let $C$ be any Euler system for $F$. Then Theorem
\ref{coreker} tells us that the vertex cocycle of $\gamma$ in $Tch(P)$ is an
element of $\ker M(C,P)$. This vertex cocycle is $V_{1}$, considered as a set
of edges in $Tch(P)$.

As $V_{1}\in\ker M(C,P)$, the columns of $M(C,P)$ corresponding to elements of
$V_{1}$ sum to $0$. These columns are the columns of $IAS(\mathcal{I}(C))$
corresponding to transitions of $P$ at vertices of $V_{1}$, so they correspond
to a subset of $\tau(\gamma)$. As $V_{1}\not =\emptyset$, it follows that
$\tau(\gamma)$ is dependent in $M_{\tau}(F)$.
\end{proof}

Theorem \ref{smallcir2} of the introduction follows readily. If $\gamma$ is
not an Euler circuit of a connected component, then Corollary \ref{corekercor}
tells us that $\tau(\gamma)$ contains a transverse circuit of $M_{\tau}(F)$.
If $\gamma$ is an Euler circuit of a connected component, then that connected
component has a non-Euler circuit $\gamma^{\prime}$; necessarily $\left\vert
\tau(\gamma^{\prime})\right\vert \leq\left\vert \tau(\gamma)\right\vert $, and
Corollary \ref{corekercor} tells us that $\tau(\gamma^{\prime})$ contains a
transverse circuit of $M_{\tau}(F)$.

%

\begin{figure}
[ptb]
\begin{center}
\includegraphics[
trim=1.331449in 8.021982in 0.671247in 1.204012in,
height=1.3638in,
width=4.8957in
]%
{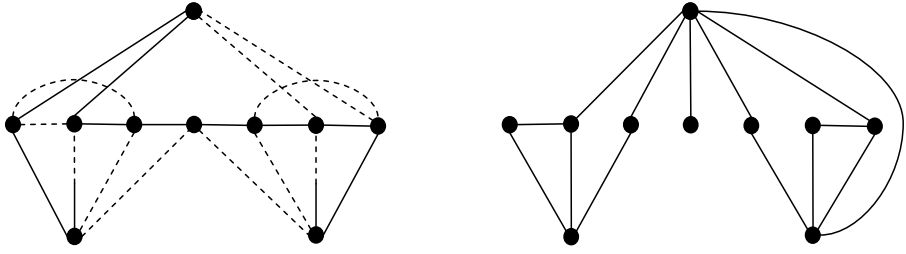}%
\caption{An Euler circuit and its interlacement graph.}%
\label{circmf2}%
\end{center}
\end{figure}

By the way, it can happen that a transverse circuit of $M_{\tau}(F)$ is
strictly smaller than every circuit in $F$. See Figure \ref{circmf2} for an
example. $F$ is a simple graph, so it has no circuit of size $<3$. But the
interlacement graph of the indicated Euler circuit $C$ has a transverse
circuit of size two: if $v$ is the vertex at the top of the figure and $w$ is
the vertex pendant on $v$ in $\mathcal{I}(C)$ then the columns of
$IAS(\mathcal{I}(C))$ representing $\phi_{C}(v)$ and $\chi_{C}(w)$ are equal,
so $\{\phi_{C}(v)$, $\chi_{C}(w)\}$ is a transverse circuit.

Before concluding this section we should explain the connection between
isotropic minors of circle graphs (as defined in Definition \ref{isominor})
and detachments of 4-regular graphs. Let $F$ be a 4-regular graph, and suppose
$t\in\mathfrak{T}(F)$ is a transition at a vertex $v$. Then the
\emph{detachment} of $F$\ along $t$ is the 4-regular graph $F^{\prime}$
obtained from $F$ by removing $v$ and forming two new edges from the four
half-edges incident at $v$, pairing together the half-edges according to
$t$. If $F$ has a loop at $v$, detachment along $t$ may result in one or
two \textquotedblleft free edges\textquotedblright\ that have no incident
vertex; any such free edge is simply discarded.

If $t$ is not a loop of $M_{\tau}(F)$, $F$ has an Euler system $C$ with
$\phi_{C}(v)=t$. As illustrated in Figure \ref{circmf10}, $F^{\prime}$ then
inherits an Euler system $C^{\prime}$ directly from $C$. Clearly
$\mathcal{I}(C^{\prime})$ is the induced subgraph of $\mathcal{I}(C)$ obtained
by removing $v$. Consequently $IAS(\mathcal{I}(C^{\prime}))$ is the matrix
obtained from $IAS(\mathcal{I}(C))$ by removing the $v$ row and all three
columns corresponding to $v$. As the only nonzero entry of the $\phi_{C}(v)$
column is a 1 in the $v$ row, the effect of these removals is that
\[
M[IAS(\mathcal{I}(C^{\prime}))]=(M[IAS(\mathcal{I}(C))]/\phi_{C}(v))-\chi
_{C}(v)-\psi_{C}(v)\text{.}%
\]
That is, $M_{\tau}(F^{\prime})$ is the isotropic minor of $M_{\tau}(F)$
obtained by contracting $t$.

If $t$ is a loop of $M_{\tau}(F)$, then $M_{\tau}(F)$ is the direct sum of
$M_{\tau}(F^{\prime})$ and the restriction of $M_{\tau}(F)$ to the vertex
triple of $v$. Consequently, $M_{\tau}(F^{\prime})$ is again the isotropic
minor of $M_{\tau}(F)$ obtained by contracting $t$.

%
\begin{figure}
[ptb]
\begin{center}
\includegraphics[
trim=1.336547in 7.762249in 0.803797in 1.075247in,
height=1.6535in,
width=4.7954in
]%
{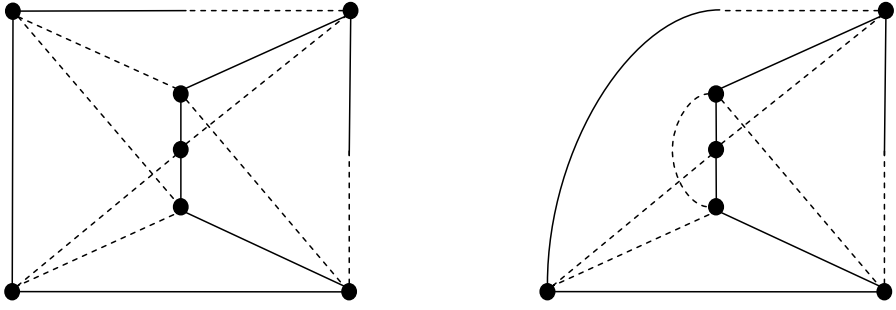}%
\caption{A detachment of a 4-regular graph inherits an Euler circuit.}%
\label{circmf10}%
\end{center}
\end{figure}

\section{Circle graphs with small transverse circuits}
\label{sec:circle_small_circuits}

Corollary \ref{corekercor} suggests that in order to characterize circle
graphs, we should obtain detailed information about small circuits in small
4-regular graphs.

\begin{proposition}
\label{smallcirc}Let $F$ be a 4-regular graph with $<9$ vertices. If $F$ has
no circuit of size $\leq3$, then $F$ is isomorphic to $K_{4,4}$.
\end{proposition}

\begin{proof}
Suppose $F$ has no circuit of size $<3$; then $F$ is a simple graph. According
to data tabulated by Meringer \cite{Me}, there are only ten simple, 4-regular
graphs of order $<9$, up to isomorphism. One is $K_{4,4}$. The other nine are
pictured in Figure \ref{circmf17}. It is evident that all nine have 3-circuits.
\end{proof}

%

\begin{figure}
[ht]
\begin{center}
\includegraphics[
trim=1.603346in 3.745206in 0.000000in 1.208415in,
height=4.5688in,
width=5.1958in
]%
{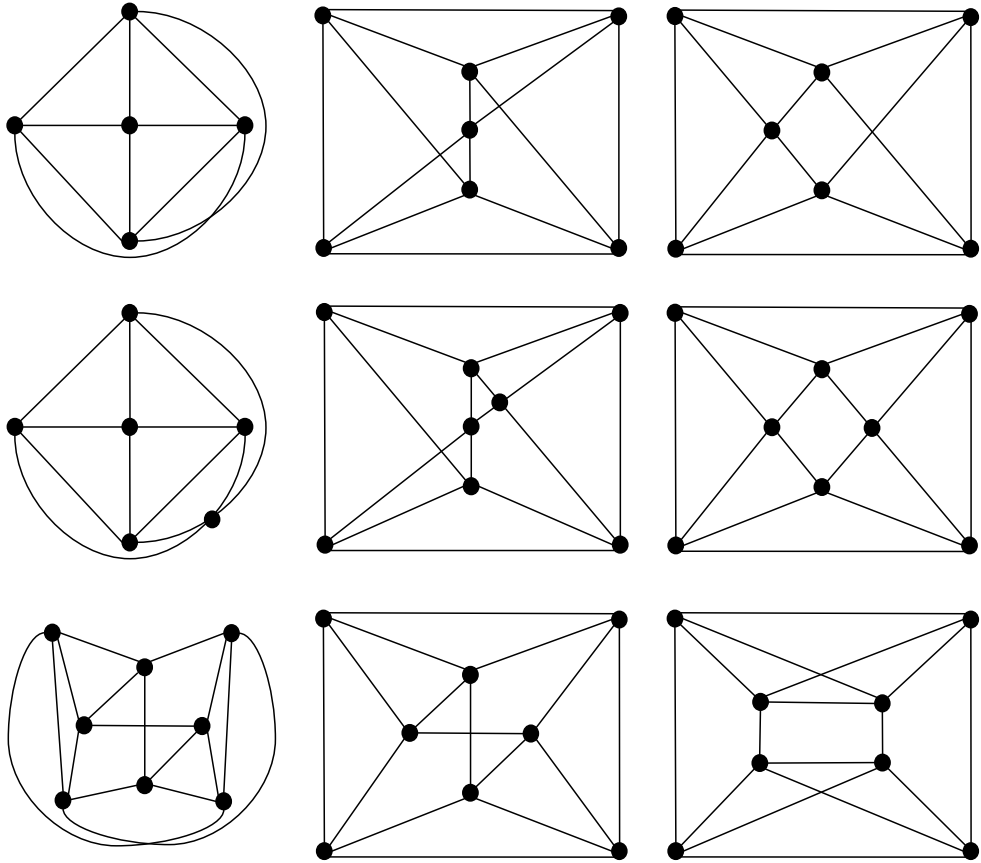}%
\caption{The simple 4-regular graphs of order $\leq8$, other than $K_{4,4}$.}%
\label{circmf17}%
\end{center}
\end{figure}

Close inspection of Figure \ref{circmf17} yields a more elaborate form of
Proposition \ref{smallcirc}, which will also be useful. In stating this
proposition we use a convenient shorthand, specifying a circuit by simply
listing the incident vertices in order.

\begin{proposition}
\label{smallcirc3}Let $F$ be a simple 4-regular graph with $<9$ vertices,
which is not isomorphic to $K_{4,4}$. Then $F$ has two distinct 3-circuits
$\gamma_{1}=v_{1}v_{2}w_{1}$ and $\gamma_{2}=v_{1}v_{2}w_{2}$, and an Euler
circuit of the form $v_{1}w_{1}v_{2}v_{1}w_{2}v_{2}$\ldots
\end{proposition}

\begin{corollary}
\label{jaeger1}A simple graph $G$ is a circle graph if and only if $G$
satisfies these two conditions.

\begin{enumerate}
\item Every transverse matroid of $G$ is cographic.

\item If an isotropic minor of $G$ of size $\leq24$ does not have a loop or a
pair of intersecting $3$-circuits, then it is isomorphic to $M_{\tau}%
(K_{4,4})$.
\end{enumerate}
\end{corollary}

\begin{proof}
Suppose $G$ is an interlacement graph of a 4-regular graph $F$. Condition 1
follows from Corollary~\ref{touchmat}.

As discussed at the end of Section 4, the isotropic minors of $G$ are
isotropic matroids of interlacement graphs of 4-regular graphs obtained from
$F$ through detachment. Consequently, if $M$ is an isotropic minor of $G$ of
size $\leq24$ then $M$ is the isotropic matroid of a circle graph $H$ with
$\leq8$ vertices. Proposition \ref{smallcirc} tells us that if $H$ is not an
interlacement graph of $K_{4,4}$ then it is an interlacement graph of a
4-regular graph with a circuit of size $\leq3$. Theorem \ref{smallcir2} tells
us that $H$ has a transverse circuit of size $\leq3$.

If $M[IAS(H)]$ has a loop, then of course condition 2 is satisfied. Suppose
$M[IAS(H)]$ has no loop; then every vertex triple is a 3-circuit. If $H$ has a
transverse 2-circuit, then $H$ has two vertices $v$ and $w$ whose vertex
triples include elements that are parallel in $M[IAS(H)]$; there is a
3-circuit obtained by replacing one element of the vertex triple of $v$ with a
parallel from the vertex triple of $w$. If $H$ has no transverse circuit of
size $\leq2$ then it must have a transverse 3-circuit, which certainly
intersects the three corresponding vertex triples.

Suppose conversely that $G$ satisfies both conditions. Condition 1 guarantees
that no isotropic minor of $M[IAS(G)]$ is isomorphic to $M[IAS(BW_{3})]$,
which has the Fano matroid and the dual Fano matroid as non-cographic
transverse matroids. Condition 2 guarantees that no isotropic minor of $G$ is
isomorphic to $M[IAS(W_{5})]$ or $M[IAS(W_{7})]$, neither of which is
isomorphic to $M_{\tau}(K_{4,4})$ or has a transverse circuit of size $\leq3$.
It follows from Theorem \ref{bouchet} that $G$ is a circle graph.
\end{proof}

Conceptually, condition 1 of Corollary \ref{jaeger1} may seem to be more
interesting than condition 2, because cographic matroids are well-studied and
the fact that transverse matroids of circle graphs are cographic is explained
by their connection with touch-graphs. However, we have not been able to
formulate a matroidal characterization of circle graphs that references only
broad properties like \textquotedblleft cographic.\textquotedblright\ Instead,
Proposition \ref{smallcirc3} allows us to replace condition 1 of Corollary
\ref{jaeger1} with a requirement involving 3-circuits.

\begin{corollary}
\label{jaeger7}A simple graph $G$ is a circle graph if and only if every
isotropic minor $M$ of $G$, of size $\leq24$, satisfies at least one of the
following conditions:

\begin{enumerate}
\item $M$ has a transverse circuit of size $\leq2$.

\item $M$ has a pair of transverse 3-circuits that are not both contained in
any one transversal of $M$.

\item $M$ is isomorphic to $M_{\tau}(K_{4,4})$.
\end{enumerate}
\end{corollary}

\begin{proof}
Suppose $G$ is an interlacement graph of a 4-regular graph $F$, and $M$ is an
isotropic minor of $G$ of size $\leq24$. Then $M$ is the isotropic matroid of
an interlacement graph $H$ of a 4-regular graph $F^{\prime}$ of order $\leq8$.
Suppose $M$ has no loop or transverse 2-circuit, and $H$ is not an
interlacement graph of $K_{4,4}$. Then Theorem~\ref{smallcir2} tells us that $F^{\prime}$ is simple, and Proposition \ref{smallcirc3} tells us
that $F^{\prime}$ has two distinct 3-circuits $\gamma_{1}=v_{1}v_{2}w_{1}$ and
$\gamma_{2}=v_{1}v_{2}w_{2}$ that share precisely one edge. Corollary
\ref{corekercor} tells us that $\tau(\gamma_{1})$ and $\tau(\gamma_{2})$ are
transverse 3-circuits of $M$. As $\gamma_{1}$ and $\gamma_{2}$ do not involve
the same transition at $v_{1}$ or $v_{2}$, no subtransversal of $H$ contains
more than four elements of $\tau(\gamma_{1})\cup\tau(\gamma_{2})$.

For the converse, suppose $G$ satisfies the statement. Then no isotropic minor
of $G$ is isomorphic to $M[IAS(W_{5})]$ or $M[IAS(W_{7})]$, because neither of
these isotropic matroids is isomorphic to $M_{\tau}(K_{4,4})$ and neither has
any transverse circuit of size $\leq3$. $BW_{3}$ does have transverse
3-circuits: the neighborhood circuits of the degree-2 vertices are of size 3,
and the $\chi_{G}$ elements of the three degree-2 vertices also constitute a
transverse circuit. A computer search using the matroid module for Sage
\cite{sageMatroid, sage} verifies that these four are the only transverse
3-circuits of $BW_{3}$, though, and clearly their union is a transversal.
\end{proof}

Recall the notation used in the introduction: If $G$ is a simple graph, then
$\mathcal{VM}_{8}(G)$ denotes the set of vertex-minors of $G$ with 8 or fewer
vertices. Using this notation and the fact that the isotropic minors of $G$
are the isotropic matroids of vertex-minors of $G$, we may rephrase Corollary
\ref{jaeger7} as follows.

\begin{corollary}
\label{jaeger8}A simple graph $G$ is a circle graph if and only if every
$H\in\mathcal{VM}_{8}(G)$ that is not an interlacement graph of $K_{4,4}$, is
locally equivalent to a graph with a vertex of degree $0$ or $1$ or with two
adjacent degree-$2$ vertices.
\end{corollary}

\begin{proof}
We first prove the if direction. Assume the right-hand side of the equivalence
holds. We invoke Corollary \ref{jaeger7}. Let $M[IAS(H)]$ be an isotropic
minor of $G$ of size $\leq24$. We have $H\in\mathcal{VM}_{8}(G)$. If $H$ is an
interlacement graph of $K_{4,4}$, then by Theorem \ref{theory} $M[IAS(H)]$
satisfies condition 3 of Corollary \ref{jaeger7}. If $H$ is locally equivalent
to a graph with a vertex of degree $0$ or $1$, then by Theorem \ref{smallcir1}
$M[IAS(H)]$ satisfies condition 1 of Corollary \ref{jaeger7}. Finally, if $H$
is locally equivalent to a graph with two adjacent degree-$2$ vertices, then
by Theorem \ref{smallcir3} $M[IAS(H)]$ satisfies condition 2 of Corollary
\ref{jaeger7}. Thus, by Corollary \ref{jaeger7}, $G$ is a circle graph.

For the converse, suppose $G$ is a circle graph and $H\in\mathcal{VM}_{8}(G)$
is not an interlacement graph of $K_{4,4}$ and not locally equivalent to a
graph with a vertex of degree $0$ or $1$. Then $H$ is isomorphic to an interlacement graph
of a 4-regular graph $F$ pictured in Figure \ref{circmf17}. As noted in
Proposition \ref{smallcirc3}, $F$ has an Euler circuit $C$ of the form
$v_{1}w_{1}v_{2}v_{1}w_{2}v_{2}\ldots$. Then $v_{1}$ and $v_{2}$ are adjacent
degree-$2$ vertices in $\mathcal{I}(C)$, and $H$ is locally equivalent to a graph isomorphic to $\mathcal{I}(C)$.
\end{proof}

\section{$K_{4,4}$ vs. $W_{7}$}

The results of the preceding section leave us with the task of distinguishing $W_{7}$ from an interlacement graph of $K_{4,4}$. This task is a bit more difficult than one might expect. For instance, it turns out that both $W_{7}$ and an interlacement graph of $K_{4,4}$ have 42 transverse 4-circuits, 168 transverse 6-circuits and no other transverse circuit of size $\leq7$. Nevertheless, there are several ways to verify that $W_{7}$ is not an interlacement graph of $K_{4,4}$. We mention four in this section, and several more in the next section.

\subsection{A distinctive transverse matroid of $W_{7}$}

Notice that the transverse matroid $\{\phi_{1}$, $\phi_{2}$, $\chi_{3}$,
$\phi_{4}$, $\phi_{5}$, $\psi_{6}$, $\psi_{7}$, $\phi_{8}\}$ of $W_{7}$
contains only two circuits, $\{\phi_{1}$, $\phi_{2}$, $\chi_{3}$, $\phi_{4}\}$
and $\{\phi_{5}$, $\psi_{6}$, $\psi_{7}$, $\phi_{8}\}$. (Here vertex 1 is the
central vertex of the wheel, and the outer vertices are numbered consecutively
around the rim.) In contrast, it turns out that no transverse matroid of an
interlacement graph of $K_{4,4}$ contains precisely two circuits. As a first
step in verifying this assertion, note that if a matroid has precisely two
circuits then the circuit elimination property guarantees that the two
circuits must be disjoint.%

\begin{theorem}
\label{smallcirc2}The smallest transverse circuits of $M_{\tau}(K_{4,4})$ are
of size 4, and there is no non-transverse 4-circuit. If a transverse matroid
$M$ of $M_{\tau}(K_{4,4})$ contains two disjoint circuits then they are of
size 4, and $M$ contains six transverse circuits of size 4.
\end{theorem}

\begin{proof}
Let $V(K_{4,4})=\{1,2,3,4,a,b,c,d\}$, and consider the Euler circuit $C$ given
by the double occurrence word $a1b2c3b4a3d4c1d2$. If $I$ is the $V(\mathcal{I}(C))\times V(\mathcal{I}(C))$
identity matrix then $IAS(\mathcal{I}(C))$ is the $V(\mathcal{I}(C))\times W(\mathcal{I}(C))$ matrix
\[
\noindent%
\begin{pmatrix}
&  & 0 & 1 & 0 & 0 & 1 & 0 & 0 & 1 & 1 & 1 & 0 & 0 & 1 & 0 & 0 & 1\\
&  & 1 & 0 & 0 & 0 & 1 & 1 & 0 & 0 & 1 & 1 & 0 & 0 & 1 & 1 & 0 & 0\\
&  & 0 & 0 & 0 & 1 & 1 & 1 & 0 & 0 & 0 & 0 & 1 & 1 & 1 & 1 & 0 & 0\\
&  & 0 & 0 & 1 & 0 & 1 & 0 & 0 & 1 & 0 & 0 & 1 & 1 & 1 & 0 & 0 & 1\\
I &  & 1 & 1 & 1 & 1 & 0 & 0 & 1 & 0 & 1 & 1 & 1 & 1 & 1 & 0 & 1 & 0\\
&  & 0 & 1 & 1 & 0 & 0 & 0 & 1 & 0 & 0 & 1 & 1 & 0 & 0 & 1 & 1 & 0\\
&  & 0 & 0 & 0 & 0 & 1 & 1 & 0 & 1 & 0 & 0 & 0 & 0 & 1 & 1 & 1 & 1\\
&  & 1 & 0 & 0 & 1 & 0 & 0 & 1 & 0 & 1 & 0 & 0 & 1 & 0 & 0 & 1 & 1
\end{pmatrix}
\text{.}%
\]
The rows are indexed in $1,2,3,4,a,b,c,d$ order, as are the columns in each
part of $%
\begin{pmatrix}
I & \mathcal{I}(C) & I+\mathcal{I}(C)
\end{pmatrix}
$.

Computer programs and visual inspection indicate that no single column is
identically 0, no two columns sum to 0, and every set of three columns that
sum to 0 corresponds to a vertex triple. The circuits of size 4 are all
transverse circuits, and there are 42 of them. These circuits fall into two sets:

\begin{itemize}
\item For each of the six 2-element subsets $\{i,j\}$ of $\{1,2,3,4\}$ and
each of the six 2-element subsets $\{x,y\}$ of $\{a,b,c,d\}$, there is a
transverse circuit that includes the transitions determined by the single
transitions $ixj$, $xjy$, $jyi$ and $yix$. For instance, $\{\phi_{C}(1)$,
$\phi_{C}(a)$, $\chi_{C}(2)$, $\phi_{C}(b)\}$ and $\{\chi_{C}(3)$, $\chi
_{C}(c)$, $\phi_{C}(4)$, $\phi_{C}(d)\}$ are transverse circuits of this type.

\item For each of the six 2-element subsets $\{i,j\}$ of $\{1,2,3,4\}$, there
is a transverse circuit that includes the transitions determined by the single
transitions $iaj$, $ibj$, $icj$ and $idj$. Notice that $\{i,j\}$ and its
complement in $\{1,2,3,4\}$ determine the same four transitions, so there are
only three transverse circuits of this type. Similarly, there are three
transverse circuits corresponding to 2-element subsets of $\{a,b,c,d\}$. For
instance, $\{i,j\}=\{1,2\}$ and $\{3,4\}$ both yield $\{\phi_{C}(a)$,
$\phi_{C}(b)$, $\chi_{C}(c)$, $\phi_{C}(d)\}$ and $\{x,y\}=\{a,b\}$ and
$\{c,d\}$ both yield $\{\phi_{C}(1)$, $\chi_{C}(2)$, $\chi_{C}(3)$, $\phi
_{C}(4)\}$.
\end{itemize}

\begin{figure}
[ptb]
\begin{center}
\includegraphics[
trim=2.541393in 8.642698in 1.472496in 1.616723in,
height=0.5872in,
width=3.3918in
]%
{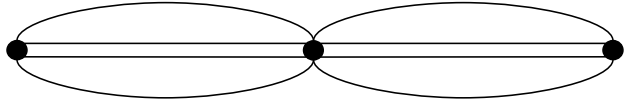}%
\caption{The cocycle matroid of this graph is isomorphic to a transverse
matroid of $W_{7}$, but not of $M_{\tau}(K_{4,4})$.}%
\label{circmf13}%
\end{center}
\end{figure}

Now, suppose a transverse matroid $M$ of $M_{\tau}(K_{4,4})$ contains two
disjoint circuits. As $M$ has 8 elements and $M_{\tau}(K_{4,4})$ has no
transverse circuit of size $<4$, the two circuits must both be of size 4. If
the two disjoint circuits are of the first type, then after re-indexing we may
suppose they arise from the circuits $1a2b$ and $3c4d$. Then $M$ also contains
transverse circuits of the first type corresponding to the circuits $1c2d$ and
$3a4b$. Moreover, $M$ contains the transverse circuits of the second type
corresponding to $\{1,2\}$ (or $\{3,4\}$) and $\{a,b\}$ (or $\{c,d\}$). If the
two disjoint transverse circuits are of the second type, we may presume one
corresponds to $\{1,2\}$ (or $\{3,4\}$) and the other to $\{a,b\}$ (or
$\{c,d\}$). Then $M$ contains the transverse circuits of the first type that
correspond to the circuits $1a2b$, $3c4d$, $1c2d$ and $3a4b$. Finally, it is
impossible for the two disjoint circuits to include one of the first type and
one of the second type.
\end{proof}

\begin{corollary}
\label{smalldiam}Every interlacement graph of an Euler circuit of $K_{4,4}$ is
of diameter $\leq2$.
\end{corollary}

\begin{proof}
According to\ Theorem \ref{smallcir4}, a graph of diameter $>2$ has a
subtransversal that contains two disjoint circuits, whose union contains no
other circuit.\ Theorem \ref{smallcirc2} tells us that $M_{\tau}(K_{4,4})$ has
no such subtransversal.
\end{proof}

\subsection{Interlacement graphs of $K_{4,4}$ are not 3-regular}

We call a simple graph \emph{strictly supercubic} if all the vertices are of
degree $\geq3$, and there is at least one vertex of degree $>3$.\

\begin{proposition}
\label{smallcirc5}Every interlacement graph of an Euler circuit of $K_{4,4}$
is strictly supercubic.
\end{proposition}

\begin{proof}
Let $G$ be an interlacement graph of $K_{4,4}$; then $M[IAS(G)]\cong M_{\tau
}(K_{4,4})$, so Theorem \ref{smallcirc2} tells us that the smallest transverse
circuits in $M[IAS(G)]$ are of size 4. By Theorem \ref{smallcir1}, this
implies that all the vertices of $G$ are of degree $\geq3$. If any vertex is of degree $>3$ we are done, so we may suppose $G$ is 3-regular.

If $G$ is not connected it has two components, each isomorphic to $K_4$. But then if $v$ and $w$ are adjacent vertices of $G$, $\{\psi_{G}(v),\psi_{G}(w)\}$ is a transverse 2-circuit of $M[IAS(G)]$, contradicting Theorem \ref{smallcirc2}. We conclude that $G$ is connected.

Let $v_{1}$ and $v_{2}$ be nonadjacent vertices of $G$. If they share all
their neighbors, then $\{\chi_{G}(v_{1}),\chi_{G}(v_{2})\}$ is a transverse
2-circuit of $M[IAS(G)]$. Theorem \ref{smallcirc2} tells us that this is
impossible. On the other hand, if they share no neighbor then the diameter of $G$ is $>2$. Corollary~\ref{smalldiam} tells us that this is impossible.

\begin{figure}
[tb]
\begin{center}
\includegraphics[
trim=1.468247in 7.896518in 0.602423in 1.140180in,
height=1.5022in,
width=4.8464in
]%
{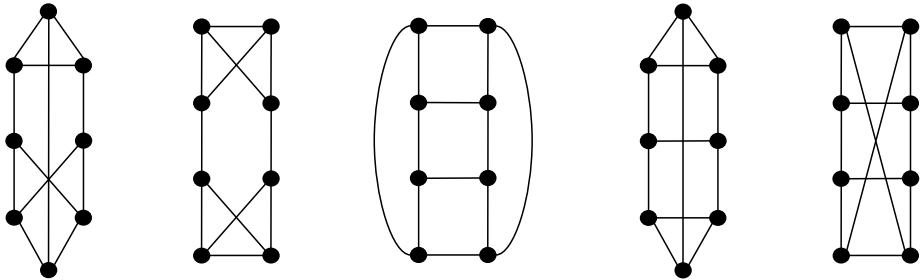}%
\caption{The connected, cubic simple graphs of order 8.}%
\label{circmf24}%
\end{center}
\end{figure}

According to Meringer \cite{Me}, there are five connected, simple 3-regular graphs with 8 vertices. They are displayed in Figure \ref{circmf24}. We observe that the first of these graphs has a pair of nonadjacent vertices that share all their neighbors, and each of the next three has a pair of nonadjacent vertices that share no neighbor. It is not immediately apparent, but the fifth graph pictured in Figure \ref{circmf24} is locally equivalent to the fourth (up to isomorphism); see
Figure \ref{circmf23} for details. (It is easy to see that the last graph in Figure \ref{circmf23} is isomorphic to the fourth graph in Figure \ref{circmf24}, if you first observe that each has precisely two 3-circuits.)
\end{proof}

\begin{figure}
[tb]
\begin{center}
\includegraphics[
trim=1.335697in 8.027484in 0.536998in 1.075247in,
height=1.4555in,
width=4.9943in
]%
{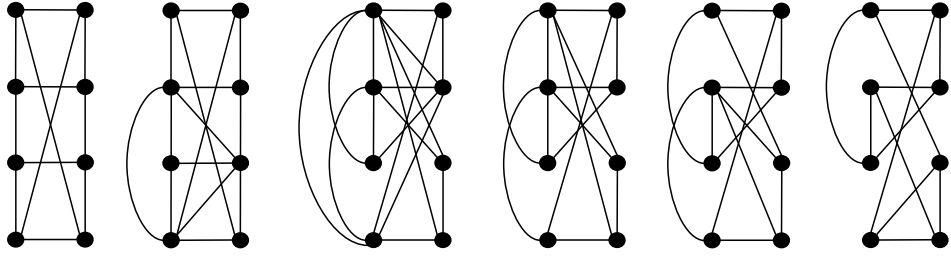}%
\caption{A local equivalence.}%
\label{circmf23}%
\end{center}
\end{figure}

By the way, there is a considerably longer proof of Proposition
\ref{smallcirc5} that may be of interest in spite of its length. Theorems
\ref{smallcir1} and \ref{smallcirc2} tell us that if $G$ is an interlacement
graph of $K_{4,4}$\ then all the transverse circuits of $G$ are of size
$\geq4$, and all the vertices of $G$ are of degree $\geq3$. Consequently, $G$
is either cubic or strictly supercubic. According to \cite{Tcub}, though,
every 3-regular circle graph has transverse circuits of size $2$; hence $G$ is
not 3-regular.

\begin{corollary}
\label{degreefive}Every interlacement graph of an Euler circuit of $K_{4,4}$
has a vertex of degree 5.
\end{corollary}

\begin{proof}
As observed at the beginning of this section, a computer search (again using
the matroid module for Sage \cite{sageMatroid, sage}) indicates that the
transverse circuits of $M_{\tau}(K_{4,4})$ of size $<9$ are all of size 4, 6
or 8. Consequently Theorem \ref{smallcir1} guarantees that if $G$ is an
interlacement graph of an Euler circuit of $K_{4,4}$ then the vertex-degrees
in $G$ are elements of the set $\{3,5,7\}$. Proposition \ref{smallcirc5}
assures us that $G$ must have at least one vertex of degree 5 or 7.

Suppose $G$ has no vertex of degree 5. If $v_{1}$ and $v_{2}$ both have degree
$7$ then the $\psi_{G}(v_{1})$ and $\psi_{G}(v_{2})$ columns of $IAS(G)$ are
the same, so $\{\psi_{G}(v_{1}),\psi_{G}(v_{2})\}$ is a transverse 2-circuit
of $G$. But $M_{\tau}(K_{4,4})$ has no transverse 2-circuit, so this cannot be
the case. That is, $G$ must have one vertex $v_{1}$ of degree 7, and seven
vertices of degree $3$. The induced subgraph $G-v_{1}$ is then $2$-regular. As
$G$ is simple, there are only two possibilities: $G-v_{1}$ is either a
7-circuit or the disjoint union of a 3-circuit and a 4-circuit. The former is
impossible, as $v_{1}$ would neighbor all the vertices on the 7-circuit and
consequently $G$ would be isomorphic to $W_{7}$. The latter is also
impossible: two nonadjacent vertices $v$ and $w$ on the 4-cycle would have
identical open neighborhoods in $G$, as $v$ and $w$ would have the same
neighbors on the 4-cycle and $v$ and $w$ would also be adjacent to $v_{1}$;
consequently $\{\chi_{G}(v),\chi_{G}(w)\}$ would be a transverse 2-circuit in
$M[IAS(G)]$.

We conclude that $G$ must have a vertex of degree 5.
\end{proof}

\subsection{A distinctive transverse matroid of $M_{\tau}(K_{4\text{,}4})$}

Here is another way to distinguish $W_{7}$ from an interlacement graph of
$K_{4,4}$.%

\begin{proposition}
\label{smallcirc4}Let $P$ be a 4-element circuit partition of $K_{4,4}$. Then
either $Tch(P)$ is obtained from a 4-cycle by doubling all edges, or $Tch(P)$
is obtained from a complete graph by doubling two non-incident edges. In
contrast, the nullity-3 transverse matroids of $W_{7}$ are all isomorphic to
the cocycle matroid of a complete graph with two non-incident edges doubled.
\end{proposition}

\begin{proof} 
The proposition has been verified with computer programs
(again using Sage \cite{sageMatroid, sage}).
\end{proof}

\begin{figure}
[ptb]
\begin{center}
\includegraphics[
trim=4.148988in 8.706531in 3.079240in 1.342683in,
height=0.7446in,
width=0.9781in
]%
{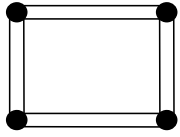}%
\caption{The cocycle matroid of this graph is isomorphic to a transverse
matroid of $M_{\tau}(K_{4,4})$, but not of $W_{7}$.}%
\label{circmf14}%
\end{center}
\end{figure}

\section{Circle graph characterizations}

As discussed in the introduction, we deduce circle graph characterizations from Bouchet's theorem and the observations of Sections 5 and 6 by combining conditions that exclude $W_5$, $BW_3$ and $W_7$ as vertex-minors, and do not exclude any circle graph. For instance, Corollary \ref{jaeger1} and Theorem \ref{smallcirc2} yield Theorem~\ref{jaeger2} in this way, and Corollaries \ref{jaeger8} and \ref{degreefive} yield Theorem \ref{jaeger6}. Other characterizations are obtained using different combinations of the properties listed below. Properties not mentioned earlier in the paper have been verified using Sage \cite{sageMatroid, sage}.

Properties of $W_5$: It is locally equivalent to a cubic graph, it has no transverse circuit of size $\leq3$, and it is not locally equivalent to any graph with a vertex of degree $\leq2$.

Properties of $BW_3$: It has some non-cographic transverse matroids, and it is not locally equivalent to a graph with either a vertex of degree $\leq 1$ or adjacent vertices of degree 2.

Properties of $W_7$: It is locally equivalent to a 3-regular graph of diameter 3, it is not locally equivalent to any graph with a vertex of degree $\leq2$, it has no transverse circuit of size $\leq3$, it has the cocycle matroid of the graph of Figure \ref{circmf13} as a transverse matroid, it does not have the cocycle matroid of the graph of Figure \ref{circmf14} as a transverse matroid, it has 42 transverse matroids of nullity 3, and its isotropic matroid has 336 automorphisms.

Properties of a circle graph of order $\leq8$ not associated with $K_{4,4}$: It has transverse circuits of size $\leq3$, it has cographic transverse matroids, and it is locally equivalent to a graph with a vertex of degree $\leq1$ or a pair of adjacent degree-2 vertices.

Properties of a circle graph associated with $K_{4,4}$: It has cographic transverse matroids, it has a vertex of degree 5, its diameter is $\leq2$, it does not have the cocycle matroid of the graph of Figure \ref{circmf13} as a transverse matroid, it does not have a transverse matroid with two disjoint circuits and no other circuit, it has the cocycle matroid of the graph of Figure \ref{circmf14} as a transverse matroid, it has 45 transverse matroids of nullity 3, and its isotropic matroid has 1152 automorphisms.

\section{Bipartite circle graphs}
\label{sec:bipartite_circle}

In this section we discuss Theorem \ref{bipartite} of the introduction, and
several associated results. These results differ from the characterizations of
general circle graphs discussed above in that they do not rely on\ Bouchet's
theorem. Their foundation is an earlier result of de Fraysseix \cite{F}; a
proof is included for the reader's convenience.

\begin{proposition}
\label{deF}(\cite[Proposition 6]{F}) Let $G$ be a bipartite simple graph, with
vertex classes $V_{1}$ and $V_{2}$. Then $G$ is a circle graph if and only if
the transverse matroid $\phi_{G}(V_{1})\cup\chi_{G}(V_{2})$ is a planar matroid.
\end{proposition}

\begin{proof}
Let the adjacency matrix of $G$ be
\[
A=%
\begin{pmatrix}
0 & A_{12}\\
A_{21} & 0
\end{pmatrix}
.
\]
Then the transverse matroids $M_{1}=\phi_{G}(V_{1})\cup\chi_{G}(V_{2})$ and
$M_{2}=\chi_{G}(V_{1})\cup\phi_{G}(V_{2})$ are represented by $%
\begin{pmatrix}
I_{1} & A_{12}%
\end{pmatrix}
$ and $%
\begin{pmatrix}
A_{21} & I_{2}%
\end{pmatrix}
$, where $I_{1}$ and $I_{2}$ are identity matrices. As $A_{12}$ and $A_{21}$
are transposes of each other, $M_{1}$ and $M_{2}$ are dual matroids.

If $G$ is a circle graph then Corollary~\ref{touchmat} tells us that $M_{1}$ and
$M_{2}$ are both cographic. They are each other's duals, so they must be planar.

Suppose conversely that $M_{1}$ is a planar matroid; then $M_{2}$ is planar
too. Let $H$ be a plane graph whose cycle matroid is $M_{1}$. We presume that
$H$ is connected, its dual $H^{\ast}$ is connected, and $H$ and $H^{\ast}$ may be drawn together in the plane in such a way that all their edges are
smooth curves, in general position. We identify $E(H)$ and $E(H^{\ast})$ with
$V(G)$.

$H$ has a spanning tree $T$ corresponding to the basis $\phi_{G}(V_{1})$ of
$M_{1}$. Choose an $\varepsilon$-neighborhood $D$ of $T$; $D$ is homeomorphic
to a disk, so its boundary $\partial D$ is homeomorphic to a circle. As long
as $\varepsilon$ is sufficiently small, each element $e\in E(H)-E(T)$
intersects $D$ in two short arcs. Use $e$ to label the end-points of these
arcs on $\partial D$. Also choose two points of $\partial D$ for each edge of
$T$, one on each side of the midpoint; label these two points with that edge.
If we trace the circle $\partial D$ and read off the labeled points in order
then we obtain a double occurrence word $W$, in which each edge of $H$ appears twice.

Consider an edge $e\in E(T)$, and write $W$ as $eW_{1}eW_{2}$. Let $\gamma$ be
a simple closed curve obtained by following $\partial D$ from one point
labeled $e$ to the other point labeled $e$, and then crossing through $D$ back
to the first point labeled $e$. Clearly then $\gamma$ encloses one connected
component of $T-e$. If $e^{\prime}$ is an edge of the connected component of
$T-e$ enclosed by $\gamma$, then both points of $\partial D$ labeled
$e^{\prime}$ appear on $\gamma$. If $e^{\prime}$ is an edge of the connected
component of $T-e$ not enclosed by $\gamma$, then neither point of $\partial
D$ labeled $e^{\prime}$ appears on $\gamma$. Either way, we see that $e$ and
$e^{\prime}$ are not interlaced with respect to $W$.

Now consider an edge $e\in E(H)-E(T)$, and write $W$ as $eW_{1}eW_{2}$.
Suppose $e^{\prime}\neq e\in E(H)-E(T)$. If $e^{\prime}$ appears precisely
once in $W_{1}$, then it is not possible to draw $e$ and $e^{\prime}$ without
a crossing outside $D$; as there is no such crossing, it cannot be that
$e^{\prime}$ appears precisely once in $W_{1}$. That is, $e$ and $e^{\prime}$
are not interlaced with respect to $W$.\ Considering the preceding paragraph,
we deduce that the interlacement graph $\mathcal{I}(W)$ is bipartite, with
$E(T)$ and $E(H)-E(T)$ as vertex-classes.

Again, let $e\in E(H)-E(T)$, and write $W$ as $eW_{1}eW_{2}$. Remove all
appearances of edges not in $T$ from $W_{1}$ and $W_{2}$; the result is two
walks in $T$ connecting the end-vertices of $e$. If we remove the edges that
appear twice in either of these walks, we must obtain the unique path in $T$
connecting the end-vertices of $e$. We conclude that the fundamental circuit
of $e$ with respect to $T$ in $H$ coincides with the\ closed neighborhood of
$e$ in $\mathcal{I}(W)$. As $T$ is a spanning tree of $H$ corresponding to the
basis $\phi_{G}(V_{1})$ of $M_{1}$, the fundamental circuit of $e$ with
respect to $T$ in $H$ is the same as the closed neighborhood of $e$ in $G$.

As $G$ and $\mathcal{I}(W)$ are both bipartite simple graphs with a
vertex-class corresponding to $V_{1}=E(T)$, it follows that $G=\mathcal{I}(W)$.
\end{proof}

Note that in the situation of Proposition \ref{deF}, $G$ is a special kind of
circle graph: it is an interlacement graph of a planar 4-regular graph $F$. To
construct such an $F$, start with the closed curve $\partial D$ mentioned in
the proof of Proposition \ref{deF}. $F$ has a vertex at the midpoint of each
edge of $T$, and a vertex outside $D$ on each edge of $E(H)-E(T)$; edges are
inserted so that the word $W$ corresponds to an Euler circuit of $F$. It is a
simple matter to draw $F$ in the plane using the closed curve $\partial D$ as
a guide.

Proposition \ref{deF} implies the following.

\begin{theorem}
\label{planar}Let $G$ be a simple graph. Then each of these conditions implies the others.

\begin{enumerate}
\item $G$\ is the interlacement graph of an Euler system of a planar 4-regular graph.

\item There are disjoint transversals $T_{1},T_{2}$ of $W(G)$ such that
$r(T_{1})+r(T_{2})=\left\vert V(G)\right\vert $ and
\[
M[IAS(G)]\mid(T_{1}\cup T_{2})
\]
\ is a planar matroid.

\item $G$ is locally equivalent to a bipartite graph, and all the transverse
matroids of $G$ are cographic.
\end{enumerate}
\end{theorem}

\begin{proof}
If $G$ is a circle graph associated with a planar 4-regular graph then it is
well known that $G$ is locally equivalent to a bipartite circle graph; see
\cite{RR} for instance. This and Corollary~\ref{touchmat} give us the implication
$1\Rightarrow3$.

Suppose condition 3 holds, and let $H$ be a bipartite graph locally equivalent
to $G$. As there is an induced isomorphism between the isotropic matroids of
$G$ and $H$, we may verify condition 2 for $G$ by verifying it for $H$. Let
$V(H)=V_{1}\cup V_{2}$ with $V_{1}$ and $V_{2}$ both stable sets of $H$. Then
the adjacency matrix of $H$ is%
\[
A=%
\begin{pmatrix}
0 & A_{12}\\
A_{21} & 0
\end{pmatrix}
.
\]
Consider the transversals $T_{1}=\phi_{H}(V_{1})\cup\chi_{H}(V_{2})$ and
$T_{2}=\phi_{H}(V_{2})\cup\chi_{H}(V_{1})$. Notice that $r(T_{i})=\left\vert
V_{i}\right\vert $ for each $i$, so $r(T_{1})+r(T_{2})=\left\vert
V(G)\right\vert $. Moreover, the transverse matroids $M[IAS(H)]\mid T_{1}$ and
$M[IAS(H)]\mid T_{2}$ are represented by $%
\begin{pmatrix}
I_{1} & A_{12}%
\end{pmatrix}
$ and $%
\begin{pmatrix}
A_{21} & I_{2}%
\end{pmatrix}
$, where $I_{1}$ and $I_{2}$ are identity matrices. As $A_{12}$ and $A_{21}$
are transposes, these two matroids are duals. Both are cographic by condition
3, so both are planar. Notice that
\[%
\begin{pmatrix}
I & A
\end{pmatrix}
=%
\begin{pmatrix}
I_{1} & 0 & 0 & A_{12}\\
0 & I_{2} & A_{21} & 0
\end{pmatrix}
,
\]
and consequently
\[
M[IAS(H)]\mid(T_{1}\cup T_{2})=(M[IAS(H)]\mid T_{1})\oplus(M[IAS(H)]\mid
T_{2})
\]
is a direct sum of planar matroids. It follows that $M[IAS(H)]\mid(T_{1}\cup
T_{2})$ is a planar matroid.

Suppose now that condition 2 holds, and $G$ has transversals $T_{1}$,
\thinspace$T_{2}$ as described. Let $B_{1}$ be a basis of $M[IAS(G)]\mid
T_{1}$, let $V_{1}=\{v\in v(G)\mid B_{1}$ contains an element of the vertex
triple of $v\}$, and let $V_{2}=V(G)-V_{1}$. Let $B_{2}\subseteq T_{2}$
consist of the elements of $T_{2}$ corresponding to elements of $V_{2}$. As
discussed in \cite[Section 4]{BT1}, $B=B_{1}\cup B_{2}$ is a basis of $M[IAS(G)]$ and there is a graph $H$ that is locally equivalent to $G$,
such that (a) an induced isomorphism $\beta:M[IAS(G)]\rightarrow M[IAS(H)]$
maps $B$ to $\{\phi_{H}(v)\mid v\in V(H)\}$ and (b) $H$ is bipartite with
vertex-classes $\beta(V_{1})$ and $\beta(V_{2})$. It follows from (b) that for
$i\neq j\in\{1,2\}$ the induced isomorphism maps $T_{i}$ to $\phi_{H}%
(V_{i})\cup\chi_{H}(V_{j})$. Proposition \ref{deF} and the subsequent
discussion imply that condition 1 holds.
\end{proof}

It is important to realize that although $BW_{3}$ is the only one of Bouchet's circle graph obstructions with non-cographic transverse matroids, condition 3 of Theorem \ref{planar} does not imply that a bipartite non-circle graph must have $BW_{3}$ as a vertex-minor. For instance, a computer search using Sage \cite{sageMatroid, sage} indicates that $BW_{3}$ is not a vertex-minor of the bipartite graph\ $BW_{4}$ pictured on the left in Figure \ref{circmf6}. Nevertheless $BW_{4}$ is not a circle graph. One way to verify this assertion is to observe that the transverse matroid $\{\chi_{1}$, $\phi_{2}$, $\chi_{3}$, $\phi_{4}$, $\chi_{5}$, $\phi_{6}$, $\chi_{7}$, $\phi_{8}$, $\phi_{9}\}$ is not cographic. Indeed, this transverse matroid is isomorphic to the cycle matroid $M(K_{3,3})$; an isomorphism is indicated by the labels $1, 2, \ldots, 9$ in the figure. (Notice that the labels denote vertices of $BW_{4}$ and edges of $K_{3,3}$.) Another way to verify that $BW_{4}$ is not a circle graph is to obtain $W_{5}$ as a vertex-minor; this can be done by performing local complementations with respect to vertices 2, 4 and 7, and then removing them.

\begin{figure}
[ptb]
\begin{center}
\includegraphics[
trim=2.141194in 7.627981in 1.206546in 1.340482in,
height=1.5532in,
width=3.8891in
]%
{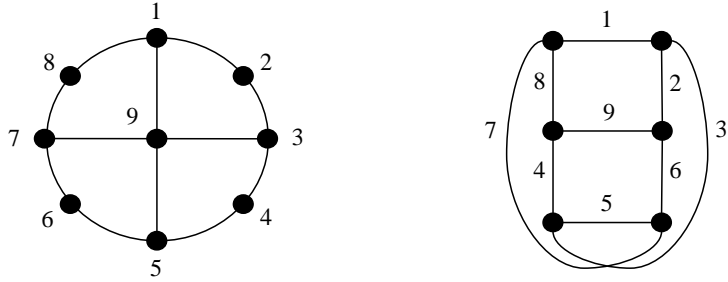}%
\caption{$BW_{4}$ and $K_{3,3}$.}%
\label{circmf6}%
\end{center}
\end{figure}

Observe that the implication $3\Rightarrow1$ of Theorem \ref{planar} tells us that the converse of Corollary~\ref{touchmat} holds for graphs that are locally equivalent to bipartite graphs. This observation yields the equivalence between the first two properties mentioned in Theorem \ref{bipartite} of the introduction. The next result implies that the third property mentioned in Theorem \ref{bipartite} is equivalent to the the first two.

\begin{corollary}
Let $G$ be a bipartite simple graph. Then $G$ is a circle graph if and only if neither $BW_{3}$ nor $BW_{4}$ is a vertex-minor.
\end{corollary}

\begin{proof}
Of course, if $BW_{3}$ or $BW_{4}$\ is a vertex-minor then $G$ is not a circle graph.

For the converse, suppose $G$ is not a circle graph. If $V_{1}$ and $V_{2}$
are the vertex-classes of $G$ then Proposition \ref{deF} tells us that the
transverse matroid $M=\phi_{G}(V_{1})\cup\chi_{G}(V_{2})$ is not a planar
matroid. As noted in the proof of Proposition \ref{deF}, the transverse
matroid $M^{\ast}=\phi_{G}(V_{2})\cup\chi_{G}(V_{1})$ is the dual of $M$; so
of course $M^{\ast}$ is not planar either.

Suppose for the moment that $M$ is minor-minimal with regard to nonplanarity.
Then $M$ or $M^{\ast}$\ is isomorphic to one of $F_{7}$, $M(K_{3,3})$,
$M(K_{5})$, so $G$ is a fundamental graph of one of these matroids. The
fundamental graphs of a binary matroid are all equivalent under edge pivots,
so they are certainly locally equivalent; hence it suffices to verify that one
fundamental graph of each of $F_{7}$, $M(K_{3,3})$, and $M(K_{5})$ has
$BW_{3}$ or $BW_{4}$ as a vertex-minor. $BW_{3}$ is a fundamental graph of
$F_{7}$, and $BW_{4}$ is a fundamental graph of $M(K_{3,3})$. \ For $K_{5}$,
consider the fundamental graph pictured on the left in Figure \ref{circmf12}.
(The labels indicate an isomorphism between the transverse matroid $\{\phi
_{1}$, $\chi_{2}$, $\chi_{3}$, $\chi_{4}$, $\chi_{5}$, $\chi_{6}$, $\chi_{7}$,
$\phi_{8}$, $\phi_{9}$, $\phi_{10}\}$ and $M(K_{5})$.) Clearly $BW_{3}$ is a
vertex-minor, obtained by first performing local complementations at the
vertices $2$, $3$, $4$, and then removing them.

Proceeding inductively, suppose that $M$ is not minor-minimal with regard to
nonplanarity. Then there is an $m\in M$ such that $M/m$ or $M-m$\ is also
nonplanar. If $M-m$ is nonplanar, then $M^{\ast}/m=(M-m)^{\ast}$ is nonplanar.
Consequently by interchanging $V_{1}$ and $V_{2}$ if necessary, we may presume
that $M/m$ is nonplanar.

%

\begin{figure}
[ptb]
\begin{center}
\includegraphics[
trim=2.407144in 7.624680in 1.070597in 1.344884in,
height=1.5558in,
width=3.7922in
]%
{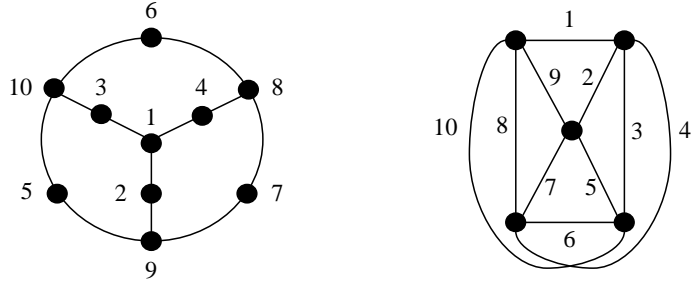}%
\caption{A fundamental graph of $K_{5}$.}%
\label{circmf12}%
\end{center}
\end{figure}

If $m=\phi_{G}(v)\in\phi_{G}(V_{1})$ then $M/m$ is a transverse matroid of
$G-v\,$, so applying the inductive hypothesis to $G-v$ implies that $BW_{3}$
or $BW_{4}$ is a vertex-minor of $G$. If $m=\chi_{G}(v)\in\chi_{G}(V_{2})$ is
a loop of $M$, then $v$ is an isolated vertex of $G$, and again we may apply
the inductive hypothesis to $G-v$. If $m=\chi_{G}(v)\in\chi_{G}(V_{2})$ is not
a loop of $M$, then $v$ has a neighbor $w\in V_{1}$. In this case the edge
pivot $G^{vw}$ is a bipartite graph with vertex classes $V_{1}\Delta\{v,w\}$
and $V_{2}\Delta\{v,w\}$, and an induced isomorphism $\beta
:M[IAS(G)]\rightarrow M[IAS(G^{vw})]$ has $\beta(\phi_{G}(w))=\chi_{G^{vw}%
}(w)$ and $\beta(\chi_{G}(v))=\phi_{G^{vw}}(v)$ \cite{Tnewnew}. As $\beta(M)$
is isomorphic to $M$, it too is nonplanar; and so is $\beta(M)/\phi_{G^{vw}%
}(v)\cong M/\chi_{G^{vw}}(v)$. As $\beta(M)/\phi_{G^{vw}}(v)$ is a transverse
matroid of $G^{vw}-v$, the inductive hypothesis implies that $BW_{3}$ or
$BW_{4}$ is a vertex-minor of $G^{vw}$, and hence of $G$.
\end{proof}

\section{Crossing numbers of 4-regular graphs}
\label{sec:cross_nr_4reg}

A simple construction indicates that every 4-regular graph can be obtained from a planar 4-regular graph through detachment. Examples appear in Figure \ref{circmf17}: the graphs in the top corner positions are detachments of the planar graphs directly below them.

\begin{theorem}
\label{detach}Every 4-regular graph is a detachment of a planar 4-regular graph.
\end{theorem}

\begin{proof}
Draw a 4-regular graph $F$ in the plane, with its edges in general position.
That is, the only failures of planarity are points where two edges cross. To
obtain a planar graph with $F$ as a detachment, replace each edge-crossing
with a vertex.
\end{proof}

This leads to yet another characterization of circle graphs:

\begin{corollary}
\label{vmb}A graph is a circle graph if and only if it is a vertex-minor of a bipartite circle graph.
\end{corollary}

\begin{proof}
If $G$ is a vertex-minor of any circle graph (bipartite or not), then $G$ itself is a circle graph. For the converse, suppose $G$ is a circle graph associated with a 4-regular graph $F$. According to Theorem~\ref{detach}, $F$ is a detachment of a plane 4-regular graph $\hat F$; then $G$ is a vertex-minor of every circle graph associated with $\hat F$, as discussed at the end of Section 4. Theorem~\ref{planar} tells us that some circle graph associated with $\hat F$ is bipartite.
\end{proof}

Consequently, a simple graph is a circle graph if and only if it is a
vertex-minor of some graph that satisfies Theorem \ref{planar}. Considering
condition 2 of Theorem \ref{planar}, one might hope that
Corollary \ref{vmb} would lead to a characterization of circle graphs using
some matroidal property of the unions of pairs of transversals. We do not know what such a property might be, though. 

We close with some comments about the connection between pairs of transverse matroids of circle graphs and a well-known measure of  non-planarity, the crossing number.

\begin{proposition}
\label{cross}Let $G$ be a simple graph. If $G$ is the interlacement graph of an Euler system of a 4-regular graph of crossing number $k$, then there are
disjoint transversals $T_{1},T_{2}$ of $W(G)$ such that $r(T_{1})+r(T_{2}%
)\leq\left\vert V(G)\right\vert +k$.
\end{proposition}

\begin{proof}
If $k=0$, the result follows from Theorem \ref{planar}. Proceeding
inductively, suppose $k\geq1$ and $F$ is a 4-regular graph of crossing number
$k$ with $G$ as interlacement graph. Let $F^{\prime}$ be the graph obtained by
replacing one crossing with a vertex $v$. Then by the inductive hypothesis
$M_{\tau}(F^{\prime})$ has disjoint transversals $T_{1}^{\prime},T_{2}%
^{\prime}$ such that $r(T_{1}^{\prime})+r(T_{2}^{\prime})\leq\left\vert
V(G)\right\vert +1+k-1$. As $M[IAS(G)]=M_{\tau
}(F)$ is an isotropic minor of $M_{\tau}(F^{\prime})$ obtained by removing the
vertex triple of $v$, $T_{1}^{\prime}$ and $T_{2}^{\prime}$ yield disjoint transversals $T_{1},T_{2}$ of $W(G)$. Contraction and deletion cannot raise ranks, so $r(T_{1})\leq r(T_{1}^{\prime})$ and $r(T_{2})\leq
r(T_{2}^{\prime})$.
\end{proof}

Proposition \ref{cross} detects the crossing numbers of some 4-regular graphs.
For instance, consider the 4-regular graph $F$ pictured in the lower left-hand
corner of Figure \ref{circmf17}. Observe that $F$ has precisely four
3-circuits, in two pairs each of which has a shared edge. No circuit partition
can include two 3-circuits that share an edge, so a circuit partition of $F$
includes at most two 3-circuits. As $F$ has 16 edges, it follows that every
circuit partition includes $\leq4$ circuits. Theorem \ref{coreker} then tells
us that every transversal of $M_{\tau}(F)$ is of rank $\geq5$, so the smallest
possible value of $r(T_{1})+r(T_{2})$ is 10. According to Proposition
\ref{cross}, this fact guarantees that the crossing number of $F$ is $\geq2$;
as the drawing in Figure \ref{circmf17} has two crossings, we conclude that
the crossing number of $F$ is 2.

Proposition \ref{cross} is not always so precise. For instance, the
reader will have little trouble finding a pair of 4-element circuit partitions
in $K_{4,4}$ that do not share any transition. The corresponding transversals
of $M_{\tau}(K_{4,4})$ are of rank 5, so they satisfy Proposition \ref{cross}
with $k=2$. But it is well known that the crossing number of $K_{4,4}$ is 4,
not 2.

Let $F$ now denote the 7-vertex 4-regular graph pictured in the middle of the
top row of Figure \ref{circmf17}. The figure suggests that the crossing number
of $F$ is 2, and it is not very hard to show that this is indeed the case.
However Figure \ref{circmf9} indicates two partitions of $E(F)$ into four
circuits. According to Theorem \ref{coreker} both of the corresponding
transversals of $M_{\tau}(F)$ have rank 4, so they satisfy the necessary
condition of Proposition \ref{cross} with $k=1$. They do not satisfy the
stronger necessary condition of Proposition \ref{onecross}, though.%

\begin{figure}
[ptb]
\begin{center}
\includegraphics[
trim=1.471646in 7.764451in 0.669548in 1.073046in,
height=1.6535in,
width=4.7945in
]%
{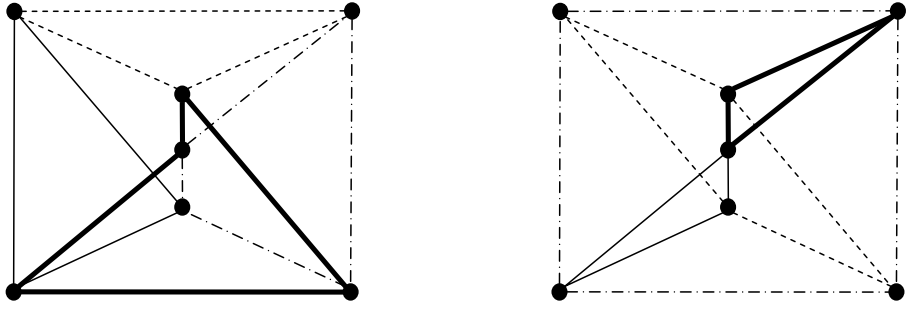}%
\caption{Two circuit partitions in a graph of crossing number 2.}%
\label{circmf9}%
\end{center}
\end{figure}

\begin{proposition}
\label{onecross}Suppose $G$ is the interlacement graph
of an Euler system of a 4-regular graph of crossing number $1$, but is not
the interlacement graph of an Euler system of any planar 4-regular graph. Then
there are disjoint transversals $T_{1},T_{2}$ of $W(G)$ such that
$r(T_{1})+r(T_{2})=\left\vert V(G)\right\vert +1$ and $M[IAS(G)]\mid(T_{1}\cup
T_{2})$ is a planar matroid.
\end{proposition}

\begin{proof}
Suppose $F$ is a 4-regular graph with an Euler system $C$ such that
$\mathcal{I}(C)=G$, and $F$ has crossing number 1. Let $F^{\prime}$ be a
planar 4-regular graph with a vertex $v$ such that detachment at $v$ yields
$F$. Let $\tau_{0}(v)$ be the transition that is detached. It is not part of
the boundary circuit of a face of $F^{\prime}$; if it were, the detachment
would be planar. $M_{\tau}(F^{\prime})$ has transversals $T_{1}^{\prime}%
,T_{2}^{\prime}$ as in Theorem \ref{planar}, and $\tau_{0}(v)$ is not included
in either $T_{1}^{\prime}$ or $T_{2}^{\prime}$; let $\tau_{1}(v)\in
T_{1}^{\prime}$ and $\tau_{2}(v)\in T_{2}^{\prime}$ be the other two
transitions at $v$. Then $M_{\tau}(F)=(M_{\tau}(F^{\prime})/\tau_{0}%
(v))-\tau_{1}(v)-\tau_{2}(v)$.

Let $S=T_{1}^{\prime}\cup T_{2}^{\prime}\cup\{\tau_{0}(v)\}$, and recall that%
\[
M_{\tau}(F^{\prime})\mid(T_{1}^{\prime}\cup T_{2}^{\prime})=(M_{\tau
}(F^{\prime})\mid T_{1}^{\prime})\oplus(M_{\tau}(F^{\prime})\mid T_{2}%
^{\prime}).
\]
Then the matroid $M=M_{\tau}(F^{\prime})\mid S$ has five kinds of circuits:
(i) Each circuit of $M_{\tau}(F^{\prime})\mid T_{1}^{\prime}$ or $M_{\tau
}(F^{\prime})\mid T_{2}^{\prime}$ is also a circuit of $M$. (ii) The vertex
triple $\{\tau_{0}(v)$, $\tau_{1}(v)$, $\tau_{2}(v)\}$ is a circuit of $M$.
(iii) For each circuit $\gamma_{1}$ of $M_{\tau}(F^{\prime})\mid T_{1}%
^{\prime}$ that contains $\tau_{1}(v)$, $M$ also has the circuit $(\gamma
_{1}-\{\tau_{1}(v)\})\cup\{\tau_{0}(v)$, $\tau_{2}(v)\}$. (iv) For each
circuit $\gamma_{2}$ of $M_{\tau}(F^{\prime})\mid T_{1}^{\prime}$ that
contains $\tau_{2}(v)$, $M$ also has the circuit $(\gamma_{2}-\{\tau
_{2}(v)\})\cup\{\tau_{0}(v)$, $\tau_{1}(v)\}$. (v) For each pair of circuits
$\gamma_{1}$, $\gamma_{2}$ as in (iii) and (iv), $M$ has the circuit
$(\gamma_{1}-\{\tau_{1}(v)\})\cup(\gamma_{2}-\{\tau_{2}(v)\})\cup\{\tau
_{0}(v)\}$.

Let $T_{1},T_{2}$ be the transversals in $M_{\tau}(F)$ that correspond to
$T_{1}^{\prime}$ and $T_{2}^{\prime}$. Then $r(T_{1})+r(T_{2})\leq
r(T_{1}^{\prime})+r(T_{2}^{\prime})=\left\vert V(F^{\prime})\right\vert
=\left\vert V(F)\right\vert +1$. Also,%
\[
M_{\tau}(F)\mid(T_{1}\cup T_{2})=(M-\tau_{1}(v)-\tau_{2}(v))/\tau_{0}(v)
\]
has only two kinds of circuits: (i) Each circuit of $M_{\tau}(F^{\prime})\mid
T_{1}^{\prime}$ or $M_{\tau}(F^{\prime})\mid T_{2}^{\prime}$ that does not
intersect the vertex triple $\{\tau_{0}(v)$, $\tau_{1}(v)$, $\tau_{2}(v)\}$ is
a circuit of $M$. (ii) Each circuit $\gamma$ listed under (v) above yields a
circuit $\gamma-\{\tau_{0}(v)\}$ in $M$. That is to say, if $M_{1}$ and
$M_{2}$ are the matroids obtained by modifying $M_{\tau}(F^{\prime})\mid
T_{1}^{\prime}$ and $M_{\tau}(F^{\prime})\mid T_{2}^{\prime}$ by using a
single label for both $\tau_{1}(v)$ and $\tau_{2}(v)$, then $M_{\tau}%
(F)\mid(T_{1}\cup T_{2})$ is the 2-sum $M_{1}\oplus_{2}M_{2}$. As $M_{\tau
}(F^{\prime})\mid T_{1}^{\prime}$ and $M_{\tau}(F^{\prime})\mid T_{2}^{\prime
}$ are both planar matroids, so is $M_{\tau}(F)\mid(T_{1}\cup T_{2})$.

It remains only to note that Theorem \ref{planar} tells us that the inequality
$r(T_{1})+r(T_{2})\leq\left\vert V(F)\right\vert +1$ must be an equality, for
otherwise $G$ would be the interlacement graph of an Euler system of a planar
4-regular graph.
\end{proof}

We do not know whether it is possible to significantly sharpen Proposition
\ref{cross} for $k>1$. If it is possible, examples indicate that the sharpened
version must be quite different from Proposition \ref{onecross}. For the graph
$F$ pictured in Figure \ref{circmf9}, a computer search using Sage
\cite{sageMatroid, sage} finds that there are no two disjoint transversals of
$M_{\tau}(F)$ whose union is a planar matroid. The situation in $M_{\tau
}(K_{4,4})$ is even more restrictive: there are no two disjoint transversals
whose union is a regular matroid.

\subsection*{Acknowledgements}
We thank the referee for various useful comments. 

\end{document}